\documentclass[11pt]{amsart}

\usepackage[utf8]{inputenc}
\usepackage[backend=biber,style=alphabetic,maxbibnames=50]{biblatex}
\usepackage{amssymb,amsfonts}
\usepackage{amsmath,amscd}
\usepackage[all,arc]{xy}
\usepackage{enumerate}
\usepackage{mathrsfs}
\usepackage[pdftex]{graphicx}
\usepackage[T1]{fontenc}
\usepackage[usenames,dvipsnames]{color}
\usepackage[bookmarks=false]{hyperref}
\usepackage{dsfont}
\usepackage{tikz-cd}
\usetikzlibrary{automata,positioning}
\usepackage{fullpage}

\usepackage{color}
\usepackage{xcolor}

\theoremstyle{plain}
\newtheorem{thm}{Theorem}[section]
\newtheorem{cor}[thm]{Corollary}
\newtheorem{prop}[thm]{Proposition}
\newtheorem{lem}[thm]{Lemma}

\theoremstyle{definition}
\newtheorem{defn}[thm]{Definition}

\newtheorem{aDD^+m}[thm]{ADD^+endum}

\theoremstyle{remark}
\newtheorem{rmk}[thm]{Remark}

\newcommand{\bbC}{\mathbb{C}} 

\newcommand{\bbR}{\mathbb{R}} 

\newcommand*{\defeq}{\mathrel{\vcenter{\baselineskip0.5ex \lineskiplimit0pt
			\hbox{\scriptsize.}\hbox{\scriptsize.}}}%
	=}

\title{Pressure Path 
metrics on parabolic families of polynomials}
\author{Fabrizio Bianchi}
\address{Dipartimento di Matematica, Università di Pisa, Largo Bruno Pontecorvo 5, 56127 Pisa, Italy}
 \email{fabrizio.bianchi$@$unipi.it}
\author{Yan Mary He}
\address{Department of Mathematics\\
	University of Oklahoma\\
	Norman, OK 73019}
\email{he$@$ou.edu}
\date{\today}

\begin{document}

\begin{abstract}
Let $\Lambda$ be a
subfamily of 
the moduli space of degree $D\ge2$ polynomials
defined by a finite number of parabolic relations. 
Let $\Omega$ be a bounded
stable
component of $\Lambda$
with the property that all critical points
are attracted 
by
either the persistent parabolic 
cycles
or by
attracting cycles in $\mathbb C$. 
We construct a 
positive semi-definite
pressure form on $\Omega$ and show that it defines a path metric on $\Omega$. This provides a counterpart
in complex dynamics of 
the pressure metric
on cusped Hitchin components recently studied by Kao and
Bray-Canary-Kao-Martone.
\end{abstract}

\maketitle

\section{Introduction}
Let $S$ be a closed surface of genus at least 2. The Teichm\"uller space $T(S)$ of $S$, which parametrizes the hyperbolic structures on $S$, carries a number of natural metrics defined from different perspectives, e.g., the Teichm\"uller metric, the Weil-Petersson metric, and the Thurston metric; see \cite{Hubbard06,ImaTan}. In \cite{Bridgeman10} and \cite{McMullen08}, Bridgeman and McMullen have respectively shown that the Weil-Petersson metric on $T(S)$ can be reconstructed via thermodynamic formalism.
More precisely, they proved that the Weil-Petersson metric is a constant multiple of the so-called {\it pressure metric}. 

From the perspective of Sullivan's dictionary \cite{Sullivan85},
the space of degree $D\ge 2$ Blaschke products $\mathcal{B}_D$ can be viewed as a
counterpart of the Teichm\"uller space $T(S)$ in complex dynamics; see for instance
\cite{Luo_GeoFiniteDegenI,McM_CompExpCircle,McM_RibbonRtreeHoloDyn,McM10}.
In \cite{McMullen08} McMullen,
 using thermodynamic formalism, 
introduced a counterpart of the Weil-Petersson metric on $\mathcal B_D$,
whose construction is analogous to that of the pressure metric on the Teichm\"uller space $T(S)$.
Nie and the second author constructed pressure metrics on
certain
hyperbolic components in the moduli space of degree $D\ge2$ rational maps \cite{HeNie23} and polynomial shift loci \cite{HeNie23SL}. 

If $S$ is a {\it punctured}
surface of negative Euler characteristic, the Teichm\"uller space $T(S)$ of $S$ parametrizes \emph{complete}
hyperbolic structures on $S$. Denote by $\pi_1S$ the fundamental group of $S$. A hyperbolic structure on $S$ can be identified with a discrete faithful representation $\rho : \pi_1S \to {\rm PSL}(2,\mathbb R)$ such that the element $g \in \pi_1S$ representing a puncture is mapped to a {\it parabolic} matrix, i.e.,
whose trace squares to 4.
Kao \cite{Kao20} constructed pressure metrics on such Teichm\"uller spaces.
In higher Teichm\"uller theory, Bray-Canary-Kao-Martone \cite{BCKM23} have studied pressure metrics for {\it cusped Hitchin components}, which are 
generalizations of Teichm\"uller spaces of punctured surfaces. These metrics are expected
to be the induced metric on the strata at infinity of the metric completion of the Hitchin component of a closed surface with its pressure metric \cite{BCKM23, Masur76}.

In this paper, as a natural
counterpart of cusped Hitchin components in complex dynamics, we consider 
\emph{$\Lambda$-hyperbolic
components} 
(see below)
$\Omega$ in an algebraic family $\Lambda$
of conjugacy classes of degree $D\ge 2$ 
polynomials or rational maps
defined by a finite number of parabolic relations.
We construct a 
positive semi-definite
pressure form on $\Omega$ and show that it defines a path metric on $\Omega$ whenever
$\Lambda$ is polynomial and $\Omega\subset \Lambda$ is bounded.

\subsection{Statement of results}
Let $D \ge 2$ be an integer. We denote by ${\rm rat}^{cm}_D$ (resp. ${\rm poly}^{cm}_D$) the moduli space of degree $D$ rational maps (resp. polynomials) with marked critical points $c_1,\dots,c_{2D-2}$ (resp. $c_1, \dots, c_{D-1}$), i.e., the space of M\"obius conjugacy classes of degree $D$ rational maps (resp. polynomials) whose critical points are marked.
Let $\Lambda$ be an algebraic subfamily of ${\rm rat}^{cm}_D$ (resp. ${\rm poly}^{cm}_D$) with the property 
that a (possibly empty) subset of the critical points are 
persistently attracted by a parabolic periodic point
(this property is invariant by the action of M\"obius transformations and therefore well-defined on the quotient).
We call such $\Lambda$ a \emph{parabolic subfamily} of ${\rm rat}^{cm}_D$ (resp. ${\rm poly}^{cm}_D$).
For every $\lambda\in \Lambda$,
we denote by $f_\lambda$ a rational map (resp. polynomial) in the corresponding conjugacy class. 
Up to renumbering, we denote by $c_1,\dots, c_m$ the critical points which are active on $\Lambda$.
Recall that a critical point $c$ is {\it passive} on an open subset
$\Lambda_0\subseteq \Lambda$ 
if the sequence of holomorphic functions
$\{\lambda \mapsto f_\lambda^n(c(\lambda))\}_{n\ge 1}$ is a normal family on $\Lambda_0$. Otherwise, the critical point is {\it active}. 

Let $\Omega$ be a stable component in $\Lambda$; that is, all critical points are passive on $\Omega$, which in this case is equivalent to asking that
none of the possibly active critical points $c_1, \dots, c_m$ is active on
$\Omega$.
We say that a stable component $\Omega$ is {\it $\Lambda$-hyperbolic} 
if,
for every $\lambda \in \Omega$, all the 
critical points $c_1(\lambda), \ldots, c_m(\lambda)$ 
are contained in the basin of some attracting cycle of
$f_\lambda$.
$\Lambda$-hyperbolic components 
are 
the
natural generalization
of the hyperbolic components
of ${\rm rat}^{cm}_D$ or ${\rm poly}_D^{cm}$
(which correspond to the case where $m=2D-2$ and $m=D-1$, respectively).
Bounded 
$\Lambda$-hyperbolic 
components 
(i.e., those satisfying $\Omega \Subset \Lambda$)
in ${\rm poly}_D^{cm}$
are the analogue of hyperbolic components for which all critical points are attracted to
attracting cycles in $\mathbb C$, which are the hyperbolic components in the connectedness locus.

We 
first show 
that the construction of the positive semi-definite
symmetric bilinear form $\langle \cdot, \cdot \rangle_G$ in \cite{HeNie23} can be extended to any $\Lambda$-hyperbolic stable component $\Omega$, 
where $\Lambda$ is a parabolic subfamily of 
${\rm poly}^{cm}_D$
or
${\rm rat}^{cm}_D$. Moreover, $\langle \cdot, \cdot \rangle_G$ is conformal equivalent to the pressure (pseudo) metric,
which is constructed on $\Omega$ in a similar way as McMullen \cite{McMullen08}; see Section \ref{sec_strategy_pf} for more details.

A priori, the
2-form $\langle \cdot, \cdot \rangle_G$ is only positive semi-definite. In \cite{HeNie23}, the authors gave a condition on a hyperbolic component in the moduli space of degree $D$ rational maps under which the 2-form is non-degenerate.
To obtain this,
they used a deep result of Oh-Winter \cite[Theorem 1.1]{Oh17} on the asymptotic distribution of repelling multipliers
for hyperbolic rational maps.
Since our stable $\Lambda$-hyperbolic 
components are more general than hyperbolic components,  
this result 
is unavailable in our setting.
However, we show that $\langle \cdot, \cdot \rangle_G$ still
defines a metric
on every bounded
$\Lambda$-hyperbolic component $\Omega$ of a parabolic subfamily $\Lambda$ of ${\rm poly}^{cm}_D$.
The following theorem is our main result.

\begin{thm}\label{thm_main}
Let $\Lambda$ be a
parabolic
subfamily of ${\rm poly}^{cm}_D$ and
$\Omega\Subset \Lambda$
a bounded
$\Lambda$-hyperbolic component.
Then the $2$-form $\langle \cdot, \cdot \rangle_G$
defines a metric on $\Omega$.
\end{thm}

Before we move on to discuss the ideas of the proofs, we mention related works on pressure metrics in various contexts in geometry and dynamics. In Teichm\"uller theory, pressure metrics have been studied for quasi-Fuchsian spaces of closed surfaces \cite{Bridgeman10,Bridgeman08}, Teichm\"uller spaces and quasi-Fuchsian spaces of punctured surfaces \cite{BCK23,Kao20}, and Teichm\"uller spaces of bordered surfaces \cite{Xu19}. In higher Teichm\"uller theory, pressure metrics have been studied for deformation spaces of Anosov representations \cite{BCLS15} and cusped Hitchin components \cite{BCKM23}. Pressure metrics have also been defined on the moduli space of metric graphs \cite{Kao17, PollicottSharp14} and on Culler-Vogtmann outer spaces \cite{Aougab23}. Ivrii \cite{Ivrii14} studied the metric completion of $\mathcal{B}_2$ with respect to McMullen's metric. Lee, Park, and the second author \cite{HLP23} studied the pressure metric
on the space of degree $D\ge2$ quasi-Blaschke products.

Pressure metrics are not inherently non-degenerate, 
i.e., they are not necessarily \emph{Riemannian} metrics. Indeed,
the construction of pressure metrics ensures only 
the positive semi-definiteness of the associated $2$-form. 
Although pressure metrics are known to be positive definite in most cases, 
degeneracy loci are known to exist
is several settings;
see \cite{Bridgeman10,HLP23}.

\subsection{Strategy of the proof} \label{sec_strategy_pf}
In \cite{HeNie23}, inspired by the works of Bridgeman \cite{Bridgeman10} and McMullen \cite{McMullen08}, the authors constructed a symmetric bilinear form $\langle \cdot, \cdot \rangle_G$ on any hyperbolic component 
in the moduli space 
of degree $D \ge 2$ rational maps.

Our first goal
in the present
paper is to show that their construction, with suitable modifications,
extends
to any $\Lambda$-hyperbolic component of a parabolic subfamily $\Lambda$ of ${\rm rat}^{cm}_D$
(or ${\rm poly}^{cm}_D$).
Let $\Omega$ be a $\Lambda$-hyperbolic component and
$\widetilde \Omega$
a lift of $\Omega$ in the parameter space ${\rm Rat}^{cm}_D$. Given $\lambda_0 \in \widetilde{\Omega}$,
denote by $J_{\lambda_0}$ the Julia set of $f_{\lambda_0}$. As $\Omega$ is a stable component, there exists a holomorphic motion for the Julia 
sets
as the parameter moves
in $\widetilde \Omega$. In particular, 
every
$\lambda_0 \in \widetilde \Omega$ admits
a neighborhood $U(\lambda_0)$ such that a H\"older-continuous conjugacy
$\Psi_\lambda : J_{\lambda_0} \to J_{\lambda}$ is well-defined for every $\lambda \in U(\lambda_0)$. 
We also observe that $\log|f_\lambda'| : J_\lambda \to \mathbb R$ is H\"older continuous for every $\lambda \in U(\lambda_0)$.

\medskip

Fix
$\eta \in (0, \log D)$. It is not difficult to see that there exists 
a unique number $\delta_\eta (\lambda)=\delta_\eta(f_\lambda)$ such that
the {\it pressure} $\mathcal{P}$ of the H\"older continuous function $-\delta_\eta(\lambda)\log|f_\lambda'| : J_\lambda \to \mathbb R$ satisfies 
\begin{equation*}
\mathcal{P}(-\delta_\eta(\lambda)\log|f_\lambda'|) = \eta.
\end{equation*}
We call $\delta_\eta(\lambda)$ 
the {\it Bowen
number\footnote{We avoid the name \emph{Bowen parameter} 
to avoid confusion, since, in this paper,  \emph{parameters} will mostly 
refer to elements in the space of the maps.}}
as a celebrated result of Bowen \cite{Bowen79} states that,
for a hyperbolic rational map $f$,
the Hausdorff dimension $\delta(f)$ of $J(f)$ satisfies $\mathcal{P}(-\delta(f)\log|f'|) = 0$.
We first prove (see Proposition \ref{prop_deltaanalytic})
that the
function $\delta_\eta : \widetilde{\Omega} \to \mathbb R$, sending $\lambda$ to $\delta_\eta(\lambda)$, is real-analytic,
which 
generalizes
\cite{Ruelle82} 
to our 
setting. To this end, we adapt the methods described
in \cite{SU10,UZ04real} in the context of hyperbolic semi-groups
of rational maps and of exponential maps respectively. 
The spectral gap property for the transfer operators associated to the weights $\delta_\eta(\lambda)\log|f_\lambda'|$ and their perturbations, 
as recently established in
\cite{BD23eq1,BD24eq2},
allows us to deal with the non-uniform hyperbolicity due to the presence of parabolic points and to
adapt
those arguments in our context. 
Observe that the condition on $\eta$ is required to apply these results, as the transfer operator does not have a unique isolated eigenvalue of multiplicity $1$ for $\eta=0$ as soon as a map has a parabolic periodic point.

\medskip

Once the analyticity of the map 
$\delta_\eta : \widetilde \Omega \to \mathbb R$
is established,
we
consider the map $G_{\lambda_0}: U(\lambda_0) \to \mathbb R$ given by $$G_{\lambda_0}(\lambda) \defeq \delta_\eta(\lambda){\rm Ly}_{\lambda_0}(\lambda),$$
where ${\rm Ly}_{\lambda_0} : U(\lambda_0) \to \mathbb R$ is given by
${\rm Ly}_{\lambda_0}(\lambda) \defeq \int_{J_{\lambda_0}} \log |(f_\lambda \circ \Psi_{\lambda})'(z)| d\nu(z)$ 
and
$\nu$ is the unique equilibrium state of the H\"older continuous function $-\delta_\eta(\lambda_0)\log|f_{\lambda_0}'| : J_{\lambda_0} \to \mathbb R$. As 
the function
${\rm Ly}_{\lambda_0}$ is real-analytic
by standard arguments, 
 so is $G_{\lambda_0}$. Moreover, we show that $G_{\lambda_0}$ attains a minimum at $\lambda_0$; see Section \ref{sec:2-form-1}. It follows that the Hessian $G_{\lambda_0}''(\lambda_0) : T_{\lambda_0}\widetilde{\Omega} \times T_{\lambda_0}\widetilde{\Omega} \to \mathbb R$ of $G_{\lambda_0}$ at $\lambda_0$ gives
 a 
 positive
 semi-definite
 symmetric bilinear form
on the tangent space $T_{\lambda_0}\widetilde{\Omega}$, that we denote by
$\langle \cdot, \cdot \rangle_G$ and call 
the {\it Hessian form}.
This form descends to $\Omega$; see Section \ref{sec:2-form-2}. We use the same notation for the induced 
 form, which is still positive semi-definite.

\medskip

We prove in Section \ref{sec_conformal_equiv}
that $\langle \cdot, \cdot \rangle_G$ is conformal equivalent to the {\it pressure (pseudo) metric} 
on $\Omega$,
see
Section \ref{ss:P-omega} for the definition and details.
The pressure metric was first considered by \cite{Bridgeman10,McMullen08} in the context of (quasi-)Fuchsian spaces and the moduli space of Blaschke products. 
Thanks to this equivalence, we can show
that a key 
necessary condition for the degeneration of the form
$\langle \cdot, \cdot \rangle_G$ at a given $[\lambda_0]\in \Omega$ and $\vec{v}\in T_{[\lambda_0]} \Omega$, proved in \cite{HeNie23} in the hyperbolic case, still holds in our context;
see Lemma \ref{lem_K_equation}. Namely,
if $[\lambda_0] \in \Omega$ and $\vec{v}\in T_{[\lambda_0]}\Omega$
are such that $\langle\vec v, \vec v\rangle_G=0$,
and we assume for simplicity that $\Omega$ is an open set of some $\mathbb C^N$,
we must have
\begin{equation}\label{e:intro-K}
\frac{d}{dt}\Big|_{t=0} 
S_n\big(
\log|f'_{\lambda_0 + t\vec v} \circ \Psi_{t}(x)|\big)
= K \cdot S_n
\big(
\log|f'_{\lambda_0} 
(x)|\big)
\end{equation} 
for all 
$n$-periodic
points
$x$  in the Julia set
of $f_{\lambda_0}$ and some constant $K$ independent of $n$ and $x$.
Here, $S_n(\cdot)$ denotes the $n$-th Birkhoff sum and $\Psi_t$ is the conjugation between the Julia sets of $f_{\lambda_0}$
and $f_{\lambda_0 + t\vec {v}}$
induced by the holomorphic motion.
Establishing \eqref{e:intro-K} relies on the equivalence between the Hessian and the pressure forms mentioned above and on the fact that any
coboundary is automatically continuous. Both these ingredients rely on the spectral gap from \cite{BD23eq1,BD24eq2}.
The condition \eqref{e:intro-K} is shown to be too rigid to hold in the hyperbolic case (up to some special exceptions)
in \cite{HeNie23}, by exploiting the above
mentioned result by Oh-Winter
\cite{Oh17}. 
Therefore, in that case,
the form is positive definite, and hence it defines a Riemannian metric on all but finitely many hyperbolic components in ${\rm rat}_D^{cm}$ or ${\rm poly}_D^{cm}$.

\medskip

In the present case,
in order to prove Theorem \ref{thm_main}, we directly
show that the function $d_G: \Omega \times \Omega \to \mathbb R$ given by
\begin{equation*}
d_G(x,y) \defeq \inf_\gamma \ell (\gamma),
\quad \mbox{ where } \quad
\ell (\gamma)\defeq\int_0^1 \sqrt{\langle \gamma'(t),\gamma'(t) \rangle_G} dt,
\end{equation*}
is a distance function, which implies that $\langle \cdot, \cdot \rangle_G$ defines a path metric on $\Omega$.
Here the infimum is taken over all the $C^1$-paths 
$\gamma$
connecting $x$ to $y$ in $\Omega$.
A priori, as 
$\langle \cdot, \cdot \rangle_G$ is only positive semi-definite,
the function $d_G$ is only 
a pseudo-metric. To show
that $d_G(x,y) > 0$ for any $x\neq y$ in $\Omega$, we first show that $\langle \cdot, \cdot \rangle_G$ is real-analytic on the unit tangent bundle $UT\Omega$ of $\Omega$. In particular, we prove that the pressure function
associated to
a 4-parameter family of potentials is
jointly
real analytic in 
all
the
parameters; see Proposition \ref{prop_metric_an}. The proof shares 
a
similar framework as that of Proposition \ref{prop_deltaanalytic} and again relies
on
the spectral gap property of the transfer operators and their perturbations.
Thanks to the analyticity of 
$\langle\cdot, \cdot\rangle_G$, an analysis of the possible singular sets for this form shows that,
as soon as $d_G (x,y)=0$, there exists
a
$C^1$ 
path $\gamma$
joining them with $\ell(\gamma)=0$; see Section \ref{sec_Pf_Thm}. 

\medskip

We then proceed to show that any $C^1$ 
path $\gamma:[0,1] \to \Omega$ in $\Omega$ has strictly positive length.
This is the only point in the paper where we use
that $\Lambda$ is 
a polynomial family
and $\Omega$ is a {\it bounded} $\Lambda$-hyperbolic component. Namely, we combine
the fact that the Lyapunov exponent of the measure of maximal entropy
(which describes the asymptotic value of the multipliers of the periodic points, thanks to their equidistribution
with respect to the measure of maximal entropy \cite{Lyubich82, Lyu83entropy})
is constant on $\Omega$ with the fact that \eqref{e:intro-K} should hold true at every point of a possible path with zero length, taking $\vec v$ to be the tangent vector to the path.
This forces the absolute values of the multipliers of all repelling periodic points to have a common level set, something that leads to a contradiction.

\subsection{Organization of the paper}
The paper is organized as follows. In Section \ref{sec_prelim} we 
present the preliminary results that we will need, in particular
those related to the thermodynamic formalism and the spectral properties of the transfer operators.
In Section \ref{sec_AnalyticityBowenParameter}, we prove the analyticity of the Bowen
function $\delta_\eta$. 
In Section \ref{sec_2form},
we construct the symmetric bilinear form $\langle \cdot, \cdot \rangle_G$ on $\Omega$ and show that it is conformal equivalent to the pressure form.
We prove Theorem \ref{thm_main} in Section \ref{sec_pathmetric}.

\subsection*{Acknowledgements}
The authors would like to thank the University of Pisa and the University
of Oklahoma for the warm welcome and for the excellent work conditions.
This project has received funding from
 the
 Programme
 Investissement d'Avenir
(ANR QuaSiDy /ANR-21-CE40-0016,
ANR PADAWAN /ANR-21-CE40-0012-01),
 from the MIUR Excellence Department Project awarded to the Department of Mathematics of the University of Pisa, CUP I57G22000700001,
 and from the PRIN 2022 project MUR 2022AP8HZ9{\_}001.
The first
author is affiliated to the GNSAGA group of INdAM.

\section{Preliminaries} \label{sec_prelim}
In this section, we collect basic definitions in Section \ref{sec_redLya}. In Section \ref{ss:operators} we adapt the main
results of \cite{BD23eq1,BD24eq2}
about the spectral properties 
of transfer operators and their perturbations
to our setting.
In Section \ref{sec_app_log_f'}, we deduce basic properties of the geometric potential.
Finally, we discuss extensions of analytic functions in Section \ref{ss:extension-analytic}.

\subsection{Notations and definitions} \label{sec_redLya}
Let $D\ge 2$ be an integer. Let ${\rm Rat}^{cm}_D$ (resp. ${\rm Poly}^{cm}_D$) be the space of degree $D$ rational maps (resp. polynomials) with marked critical points $c_1,
\ldots, c_{2D-2}$ (resp. $c_1, \ldots, c_{D-1}$) and ${\rm rat}^{cm}_D$ (resp. ${\rm poly}^{cm}_D$)
the moduli space obtained from ${\rm Rat}^{cm}_D$ (resp.
${\rm Poly}^{cm}_D$) by taking the quotient by all M\"obius 
(resp. affine)
conjugacies. For every 
$\lambda \in {\rm Rat}^{cm}_D$ (resp. ${\rm Poly}^{cm}_D$), we denote by $f_\lambda$ the corresponding 
rational map (resp. polynomial),
by $\mu_\lambda$ its unique measure of maximal entropy \cite{FLM83,Lyubich82,Lyu83entropy}, and by $L(\lambda) = \int \log |f'_\lambda|\mu_\lambda$ the Lyapunov exponent of $\mu_\lambda$.

\medskip

Let $\Lambda$ be an algebraic subfamily of ${\rm Rat}^{cm}_D$ (resp. ${\rm Poly}^{cm}_D$) with the property 
that a (possibly empty) subset of the critical points are 
persistently attracted by a parabolic periodic point.
We say that such $\Lambda$ is a
\emph{parabolic subfamily} of ${\rm Rat}^{cm}_D$ (resp. ${\rm Poly}^{cm}_D$).
Examples of these subfamilies are the families ${\rm  Per}_{n} (\eta)$ for $n\in \mathbb N$ and $\eta$ root of unity, which are defined as
\[
{\rm Per}_{n} (\eta) \defeq
\{ 
\lambda \colon \exists x \colon f_\lambda^n(x)=x; f'_\lambda(x)=\eta\}.
\]
We refer to \cite{Berteloot13, Milnor93geometry, Silverman07} 
for the description and the geometry of these algebraic submanifolds, and to \cite{BB09,BB11,Dujardin14,Gauthier16} 
for their distribution in the parameter space and their role in the dynamical understanding of the bifurcation phenomena.
Every parabolic subfamily is then a finite intersection
of such families (where ${\rm Rat}_D^{cm}$ or ${\rm Poly}_D^{cm}$ is thought of as the empty intersection). It is clear that the quotient by 
  M\"obius or
 affine conjugacies is well-defined on parabolic subfamilies, and we will use the same name for the images of these families by the quotient map.
By a slight abuse of notation, we can think of parabolic subfamilies of ${\rm Poly}_D^{cm}$ (resp. ${\rm poly}_D^{cm}$) also as parabolic subfamilies of ${\rm Rat}_D^{cm}$ (resp. ${\rm rat}_D^{cm}$), where we allow (and require) the further critical relation
$f_\lambda^{-1} (\{\infty\})=\{\infty\}$.

\medskip

Recall that a critical point $c$ is called {\it passive} on
an open subset
$\Lambda_0 \subseteq \Lambda$
if the sequence of holomorphic functions
$\{\lambda \mapsto f_\lambda^n(c(\lambda))\}_{n\ge 1}$ is a normal family on $\Lambda_0$. Otherwise, the critical point is called {\it active} on $\Lambda_0$.
It follows from the definition and the fact that every parabolic basin contains at least a critical point
that, in any parabolic subfamily $\Lambda$
given by $k$ parabolic relations, there are at least $k$ passive critical points on $\Lambda$.

\medskip

Recall 
\cite{Lyu83typical,MSS83} that an open subset $\Omega\subset\Lambda$ 
is in the \emph{stability locus} of $\Lambda$
if all critical points are passive on $\Omega$. We also say that the family is stable on $\Omega$.  $\Omega$ is a \emph{stable component} if it is a connected component of the stability locus.
We say that a stable component $\Omega\Subset \Lambda$
 is {\it $\Lambda$-hyperbolic} 
if,
for every $\lambda \in \Omega$, every critical point $c_j(\lambda)$ such that
$c_j$ is active on $\Lambda$
is contained in the basin of some attracting cycle 
for $f_\lambda$.
We say that $\Omega$ is a bounded $\Lambda$-hyperbolic component if we have $\Omega \Subset \Lambda$ for the topology induced by $\Lambda$.

\medskip

In the sequel, we will need the following simple facts about polynomial
bounded 
$\Lambda$-hyperbolic
components.
We refer to \cite{zhang2022dynamical, zhang2022parabolic} for some topological properties of these components in the case of cubic polynomials.

\begin{lem} \label{lem_homology}
Let $\Lambda$ be a parabolic subfamily of ${\rm poly}_D^{cm}$ 
and
 $\Omega\Subset \Lambda$
 a bounded 
$\Lambda$-hyperbolic
component.
Then
\begin{enumerate}
\item  all the critical points which are active on $\Lambda$ are contained in the basin of some attracting cycle in $\mathbb C$ for every $\lambda\in \Omega$;
\item $L(\lambda) \equiv \log D$.
\end{enumerate}
\end{lem}

\begin{proof}
If some critical point is in the basin of infinity for some $\lambda_0\in \Omega$, the same must be true for all $\lambda\in \Omega$. It is a standard fact that the map $f_{\lambda_0}$ can be deformed, staying inside $\Omega$, to make the
modulus of this critical point to become as large as desired. This contradicts the assumption $\Omega \Subset \Lambda$ and proves (1).

It follows from (1) that every critical point must have bounded orbit for every $\lambda \in \Omega$ (as it can either be attracted to a parabolic cycle or an attracting cycle in $\mathbb C$). 
Hence, the Green function at every critical point is equal to $0$, and (2) follows
from the classical
Przyticki formula
\cite{Przytycki85}.
\end{proof}

\subsection{Transfer operators}\label{ss:operators}
We fix in this section a rational map $f$ of degree $D\ge2$
and two real numbers $q>2$ and $0<\beta \leq 2$. We denote by $J(f)$ the Julia set of $f$.
By \cite{BD23eq1,BD24eq2}, there exists
a norm $\|\cdot\|_\diamond$ 
for real-valued functions on $\mathbb P^1$
satisfying
the following properties for every
$g,h\colon \mathbb P^1 \to \mathbb R$:
\begin{itemize}
\item[{\bf (N1)}] $\|g\|_{\log^q}\lesssim \|g\|_\diamond
\lesssim \|g\|_{C^\beta}$;
\item[{\bf (N2)}] $\|gh\|_\diamond\lesssim \|g\|_\diamond \|h\|_\diamond$;
\item[{\bf (N3)}] $\|f_* g\|_\diamond\lesssim \|g\|_\diamond$,
\end{itemize}
where the implicit constants are independent of $g$ and $h$.
We denote by $\|\cdot\|_{ C^\beta}$
the $\beta$-H\"older norm, and by $\|\cdot\|_{\log^q}$ the $q$-$\log$-H\"older norm, which is defined as
\[
\|g\|_{\log^q}
\defeq
\|g\|_{L^\infty} +
\sup_{x\in \mathbb P^1, r>0}
(1+ |\log r|)^q \cdot
\big(
\sup_{B(x,r)} g  - \inf_{B(x,r)} g \big)
\]
for every $g\colon \mathbb P^1\to \mathbb R$.
When working on $\mathbb P^1$, we will always use the distance ${\rm dist}_{\mathbb P^1}$ induced by the spherical metric.
Recall also that the push-forward operator $f_*$
is defined
by
\[ (f_* g) (y) \defeq \sum_{f(x)=y} g(x)\]
for every $y\in \mathbb P^1$, where the
$D$
 preimages of $y$
 are counted with multiplicity.

\medskip

By defining $\|g\|_\diamond \defeq \|\Re g\|_\diamond + \|\Im g\|_\diamond$, we see that a norm with the same properties exists also
for complex-valued functions on $\mathbb P^1$
(see also \cite[Section 5.3]{BD24eq2}).
We set
$$B_\diamond \defeq \{g: \mathbb P^1 \to \mathbb C:  \|g\|_\diamond < \infty\}.$$ 
By {\bf (N1)},
the space $(B_\diamond,\|\cdot\|_\diamond)$ is a Banach space and contains all $\beta$-H\"older
continuous
functions.
We denote by $L(B_\diamond)$ the space of bounded linear operators on $B_\diamond$.
Given a \emph{potential}
(or \emph{weight})
$\phi \in B_\diamond$, the (Ruelle-Perron-Frobenius)
\emph{transfer operator} $\mathcal{L}_\phi : B_\diamond \to B_\diamond$ is defined by
\[
\mathcal L_\phi (g) (y) \defeq \sum_{f(x)=y} e^{\phi (x)} g(x).
\]
By \cite{BD24eq2}, the norm $\|\cdot\|_\diamond$ can be chosen so that
$\mathcal L_\phi \in L(B_\diamond)$
for every 
$\phi \in B_\diamond$ 
with 
$\max \phi - \min \phi < \log D$.
More precisely,
\begin{itemize}
\item[{\bf (N4)}]
for any such $\phi 
\in B_\diamond$ 
with 
$\max \phi - \min \phi < \log D$
and every $\psi\in B_\diamond$,
the function
$t \mapsto \mathcal L_{\phi + t \psi}$ is analytic
as operators in
$L (B_\diamond)$
for $|t|$ sufficiently small.
\end{itemize}
Recall that the analyticity of the map 
$t \mapsto \mathcal L_{\phi + t \psi}$ above means that, for every $t_0$ sufficiently small, there exist operators $L_{t_0,j}\in L(B_\diamond)$ such that
$\mathcal L_{\phi+t\psi}= 
\sum_{j\geq 0} 
(t-t_0)^j
L_{t_0,j} / j!$
for every $t$ in a neighbourhood of $t_0$.
It is not difficult to see that,
for every
$t_0$
such that $\max (\phi + t_0 \psi) - \min (\phi + t_0 \psi)<\log D$,
we can take $L_{j,t_0} (\cdot)
\defeq
\mathcal L_{\phi+t_0 \psi} (\psi^j \cdot)$, and that we have
\[
\|L_{j,t_0} \|_\diamond
\leq \|\mathcal L_{\phi+t_0 \psi}\|_\diamond \cdot \|\psi^j \|_\diamond
\leq c^j 
\|\mathcal L_{\phi+t_0 \psi}\|_\diamond \cdot \|\psi \|_\diamond^j,
\]
where $c$ is the implicit constant 
in {\bf (N2)} and 
$\|\mathcal L_{\phi+t_0 \psi}\|_\diamond$ is bounded by the assumption on $t_0$.

\medskip

Recall that, for every continuous function $\phi\colon \mathbb P^1\to \mathbb R$,
the {\it topological pressure} $\mathcal{P}(\phi)$ of $\phi$ is defined by $$\mathcal{P}(\phi)  \defeq \sup_{m \in \mathcal{M}_f} \left(h_m(f) + \int \phi \, dm \right)$$
where $\mathcal{M}_f$ is the set of $f$-invariant probability measures on $\mathbb P^1$
and
$h_m(f)$ is the measure-theoretic entropy of $f$ with respect to the measure $m$. A measure $m=m(\phi) \in \mathcal{M}_\sigma$ is called an {\it equilibrium state} 
of $\phi$ if
$\mathcal{P}(\phi)  =h_m(\sigma) + \int \phi \, dm.$

\medskip

We say that a bounded linear operator $\mathcal{L} \in L(B_\diamond)$ has a {\it spectral gap} if $\mathcal{L}$ has an isolated real eigenvalue $\kappa$ of multiplicity 1
and the rest of the spectrum of $\mathcal{L}$ is contained in a disk of radius strictly smaller than $\kappa$. By \cite{BD24eq2}, the norms $\|\cdot\|_\diamond$ can be taken to also satisfy this property. Namely,
\begin{itemize}
\item[{\bf (N5)}]
for every $\phi \in B_\diamond$ with $\max \phi- \min \phi<\log D$, the transfer operator $\mathcal{L}_\phi$ has a spectral gap with respect to the norm $\|\cdot\|_\diamond$,
with the largest eigenvalue equal to $e^{\mathcal P (\phi)}$. 
Moreover, there exists a unique equilibrium state $\mu_\phi$ associated with $\phi$.
\end{itemize}

\medskip

In this paper, we will need to
consider (families of)
weights $\phi$
which are H\"older continuous on $J(f)$, but not on $\mathbb P^1$. The definitions
of pressure and equilibrium states can also be stated in this setting.
On the other hand, in order to apply the results above, we will need to modify
(or directly define) them outside of $J(f)$. This will be achieved by means of the following elementary 
lemma, which also gives a uniform control on the extension, under suitable assumptions. We give the proof for the reader's convenience.

\begin{lem}\label{l:extension}
Let $W_1 \Subset W_2 \subset\mathbb R^\ell$ be open sets
and $\{\phi_s \colon J(f)\to \mathbb R\}_{s\in W_2}$ 
a family of 
$\beta$-H\"older continuous functions with the property that $s\mapsto \phi_s$ is continuous with respect to the $\beta$-H\"older norm.
Assume also that there exists
a constant $M<+\infty$
such that $\max \phi_s < M$ for all $s\in W_2$.

Then,
there exists a family $\{\widetilde \phi_s
\colon \mathbb P^1\to \mathbb R\}_{s\in W_1}$
of
$\beta$-H\"older continuous functions with the property that $s\mapsto \widetilde \phi_s$ is continuous with respect to the $(\beta/2)$-H\"older norm
and we have
$\max\widetilde \phi_s < M$
and $\widetilde \phi_s= \phi_s$ on $J(f)$  for all $s\in W_1$.
\end{lem}

\begin{proof}
For a single $\beta$-H\"older continuous
map $\phi\colon J(f)\to\mathbb R$, the classical McShane-Whitney theorem \cite{McShane34,Whitney34}
gives a
$\beta$-H\"older continuous
extension $\tilde \phi\colon \mathbb P^1 \to \mathbb R$ with $\tilde \phi \leq \max \phi$. Recall that such extension is given by the explicit formula
\[\tilde\phi (z)
\defeq 
\max_{x \in J (f)}
\left(\phi(x) -
C
{\rm dist}_{\mathbb P^1} (x,z)^\beta \right),\]
where $C$ is a bound for the $\beta$-H\"older semi-norm of $\phi$.
Observe that
the modulus of continuity of $\tilde \phi$ is the same as that of $\phi$, i.e.,
by construction,
the $\beta$-H\"older semi-norm of $\tilde \phi$ 
is equal to the
$\beta$-H\"older semi-norm of $\phi$. In the following, for simplicity, we denote by $\|\cdot\|'_{C^\beta}$ the $\beta$-H\"older semi-norm, and observe that we
have
$\|\cdot\|_{C^\beta}=\|\cdot\|_{L^\infty}+ \|\cdot\|'_{C^\beta}$.

\medskip

Let us now consider a family $\phi_s$
as in the statement.
Consider the compact metric space $\overline{W_1}\times \mathbb P^1$, endowed with the product metric 
of $\overline{W_1}$ (with the standard Euclidean distance ${\rm dist}_E$)
and $\mathbb P^1$ (with the spherical distance),
and its compact subset $\overline{W_1} \times J(f)$. Define the function $\Phi\colon \overline {W_1}\times J(f)$
by $\Phi(s,z)= \phi_s (z)$. 
Observe that $\Phi$ is $\beta$-H\"older continuous, and let $C$ be a bound for its $\beta$-H\"older semi-norm.
Consider the McShane-Whitney extension of $\Phi$ to $\overline {W_1} \times \mathbb P^1$, which is given by
\[
\tilde \Phi(s,z) = \sup_{(t,x)\in \overline{W_1} \times J(f)} \phi_s (x) - C 
\big(({\rm dist}_E (s,t) + {\rm dist}_{\mathbb P^1}
(z,x)\big)^\beta.
\]

For $s\in W_1$, set $\tilde \phi_s (\cdot)\defeq \tilde \Phi (s, \cdot)$. The upper bound for these functions as in the statement follows from the inequality $\max \tilde \Phi\leq \max \Phi$. 

\medskip

We now
show the continuity of the family
$\{\tilde \phi_s\}$ with respect to the 
$(\beta/2)$-H\"older norm.
It follows from their
definition  that the $\tilde \phi_s$'s
satisfy
\begin{equation}
\label{eq:check-infty}
\|\tilde \phi_{s_1}-\tilde \phi_{s_2}\|_{L^\infty}\to 0
\quad
\mbox{ as }
\quad 
{\rm dist}_E (s_1,s_2)\to 0.
\end{equation}

 Given $s_1,s_2\in W_1$ and $z_1, z_2 \in \mathbb P^1$, we will show the inequality
\begin{equation}\label{eq:bound-holder}
|\phi_{s_1} (z_1)- \phi_{s_2} (z_1) - \phi_{s_1} (z_2) - \phi_{s_2} (z_2)|
\leq 2C ({\rm dist}_{\mathbb P^1} (z_1, z_2))^{\beta/2}({\rm  dist}_{E} (s_1, s_2))^{\beta/2}.
\end{equation}
Together with \eqref{eq:check-infty},
this inequality
implies 
that
we have
\[
\|\phi_{s_1}-\phi_{s_2}\|_{C^{\beta/2}}
\to 0 
\quad
\mbox{ as } 
\quad
{\rm  dist}_{E} (s_1, s_2)\to 0,
\]
and hence the desired continuity
 with respect to the $(\beta/2)$-H\"older norm.

\medskip

Observe first that, as $\|\tilde \Phi\|'_{C^\beta}\leq C$, we have
$\|\tilde \phi_s\|'_{C^\beta}
=
\|\tilde \Phi (s, \cdot)\|'_{C^\beta}
\leq C$ for all $s\in W_1$
and 
$\| 
\tilde \Phi (\cdot, z)\|'_{C^\beta}\leq C$ for all $z\in \mathbb P^1$.
We need to consider the following two cases.

\medskip

{\bf Case 1:
${\rm dist}_E (s_1, s_2) \geq {\rm dist}_{\mathbb P^1} (z_1, z_2)$.}

\medskip

In this case, we can bound the left-hand side of \eqref{eq:bound-holder} as
\[
\begin{aligned}
|\phi_{s_1} (z_1)- \phi_{s_2} (z_1) - \phi_{s_1} (z_2) - \phi_{s_2} (z_2)|
& \leq 
(
\|\tilde \Phi ({s_1}, \cdot)\|'_{C^\beta} 
+
\|\tilde \Phi ({s_2}, \cdot)\|'_{C^\beta})
({\rm dist}_{\mathbb  P^1} (z_1, z_2))^\beta\\
& \leq 2C ({\rm dist}_{\mathbb P^1} (z_1, z_2))^\beta\\
&  \leq 2C 
({\rm dist}_E (s_1, s_2))^{\beta/2}
({\rm dist}_{\mathbb P^1} (z_1, z_2))^{\beta/2}.
\end{aligned}
\]

{\bf Case 2:
${\rm dist}_E (s_1, s_2) \leq {\rm dist}_{\mathbb P^1}
(z_1, z_2)$.}

\medskip

Reversing the roles of the $s_j$'s
and $z_j$'s, 
we now 
bound the left-hand side of \eqref{eq:bound-holder} as
\[
\begin{aligned}
|\phi_{s_1} (z_1)- \phi_{s_2} (z_1) - \phi_{s_1} (z_2) - \phi_{s_2} (z_2)|
& \leq 
(
\|\tilde \Phi (\cdot, {z_1})\|'_{C^\beta} 
+
\|\tilde \Phi (\cdot, {z_2})\|'_{C^\beta})
({\rm dist}_E (s_1, s_2))^\beta\\
& \leq 2C ({\rm dist}_E (s_1, s_2))^\beta\\
&  \leq 2C 
({\rm dist}_E (s_1, s_2))^{\beta/2}
({\rm dist}_{\mathbb P^1} (z_1, z_2))^{\beta/2}.
\end{aligned}
\]

Hence,
\eqref{eq:bound-holder} holds in either case. The proof is complete.
\end{proof}

\begin{lem}\label{l:compare}
Let $\phi\colon J(f)\to \mathbb R$ 
and $\tilde \phi \colon \mathbb P^1 \to \mathbb R$
be continuous functions with
$\tilde \phi_{|J(f)}=\phi$
and
$\max \tilde \phi = \max \phi<\mathcal P(f_{|J(f)},\phi)$. 
Then $\phi$ and $\tilde \phi$ have the same equilibrium states, which are supported on $J (f)$. In particular, we have $\mathcal P(f_{|J(f)},\phi)=\mathcal P(f,\tilde\phi)$.
\end{lem}

\begin{proof}
It follows from the definitions that any equilibrium state 
for $\tilde \phi$ supported on $J(f)$ is also an equilibrium state for $\phi$. Hence, it is enough to show that we have
\[
h_\nu (f) + \int \tilde \phi d\nu < \mathcal P (f,\tilde \phi)
\]
for every invariant probability
measure $\nu$ whose support is outside $J(f)$. Any such invariant measure necessarily satisfies $h_{\nu} (f)=0$. It then follows from the assumptions that we have
\[
\int \tilde \phi d\nu \leq \max \tilde \phi
= \max \phi < \mathcal P(f_{|J(f)},\phi)\leq \mathcal P(f,\tilde \phi).
\]
The assertion follows.
\end{proof}

Combining \cite{BD24eq2} with \cite[Main Theorem]{InoquioRL12}, one can see that 
the condition $\max_{z\in\mathbb P^1}\phi-\min_{z\in\mathbb P^1} \phi<\log D$ can be weakened to $\max_{z\in\mathbb P^1} \phi < \mathcal P(\phi)$ in both
{\bf (N4)} and {\bf (N5)}
(and actually to 
$\max_{z\in J(f)} \phi < \mathcal P(\phi)$). More precisely, we have the following lemma.

\begin{lem} \label{lem_IRL}
Let $f$ be a rational map of degree $D\ge2$
and 
$\phi \colon J(f) \to \mathbb R$
a $\beta$-H\"older continuous function.
If the Lyapunov exponent of each equilibrium state of $\phi$ is strictly positive, then there
exists a $\beta$-H\"older
continuous extension $\tilde \phi$
of $\phi$
to $\mathbb P^1$
such that
the transfer operator $\mathcal{L}_{\tilde \phi}\colon B_\diamond \to B_\diamond$ has a spectral gap, and the function $t\mapsto \mathcal L_{\tilde \phi + t \psi}$
is analytic as operators in
$L (B_\diamond)$ near $t=0$ for every $\psi \in B_\diamond$. 
Here $\|\cdot\|_\diamond$ is chosen so that $\|\cdot\|_\diamond
\lesssim \|\cdot\|_{C^\beta}$ in {\bf (N1)}. 
\end{lem}

\begin{proof}
Suppose that the Lyapunov exponent of each equilibrium state of $\phi$ is strictly positive. Then,
by \cite[Main Theorem]{InoquioRL12}, the potential $\phi$ is {\it hyperbolic}, namely, there exists an integer $n \ge 1$ such that $\max_{z \in J(f)} S_n\phi < \mathcal{P}(f^n,S_n\phi)$,
where
$S_n\phi$ 
denotes
the $n$-th Birkhoff sum
$\phi + \phi \circ f + \cdots + \phi \circ f^{n-1}$ of $\phi$. 
Up to replacing $f$ by $f^n$ (which does not change the statistical properties in the statement), we can assume that $n=1$, i.e., that
we have $\max_{z\in J(f)} \phi < \mathcal P (f, \phi)$.
Let $\tilde \phi$ be a $\beta$-H\"older continuous extension of $\phi$ as in Lemma
\ref{l:extension}. In particular, we have
$\tilde \phi\leq\max_{J(f)} \phi$, which gives
$\tilde \phi \leq \mathcal P(f,\phi)= \mathcal P(f,\tilde \phi)$, where the last equality follows from Lemma \ref{l:compare}.

\medskip

We now observe that, in \cite{BD23eq1} and \cite{BD24eq2}, the assumption $\max \tilde \phi -\min \tilde \phi<\log D$ can be replaced by $\max \tilde \phi < \mathcal P(f,\tilde \phi)$ as soon as the last quantity has already been defined (this is not the case in \cite{BD23eq1,BD24eq2}, hence the need for an assumption not directly
involving the pressure). In particular, the same proof as in \cite{BD24eq2} gives the spectral gap for the action of $\mathcal L_{\tilde \phi}$ with respect to the norm $\|\cdot\|_\diamond$ chosen as in the statement, and the other properties. The assertion follows.
\end{proof}

In particular, the above lemmas allow
us to apply the machinery of \cite{BD23eq1,BD24eq2} to functions which are a priori defined only on $J(f)$ (or which are
H\"older continuous only on $J(f)$). More precisely, 
for every $g\colon J(f)\to \mathbb R$, define
\[
\|g\|_\diamond \defeq \inf \{\|\tilde g\|_\diamond
\colon 
\tilde g \colon \mathbb P^1 \to \mathbb R, 
\quad \tilde g_{|J(f)} = g\}
\]
and
\begin{equation}\label{eq:def-BJ}
B_\diamond (J) \defeq \{g\colon J(f)\to \mathbb R :  \|g\|_\diamond<\infty\}.
\end{equation}
By Lemma
\ref{l:extension}, every $\beta$-H\"older continuous function on $J(f)$ belongs to $B_\diamond(J)$. By {\bf (N1)},
every
element of $B_\diamond (J)$
is $q$-$\log$-H\"older-continuous for some $q>2$ depending on the choice of the norm
$\|\cdot\|_\diamond$. Given $\phi,g\colon J(f)\to \mathbb R$, it is clear that
we have 
$\|\mathcal L_\phi g \|_\diamond
\leq \|\mathcal L_{\tilde \phi} \tilde g\|_\diamond$
for every extensions $\tilde \phi, \tilde g$ of $\phi$ and $g$. In particular, for every $\phi\in B_\diamond (J)$
with 
$\max \phi < \mathcal P (f,\phi)$,
the operator $\mathcal L_\phi$ is bounded
and has a spectral gap on $B_\diamond (J)$,
with the
largest eigenvalue
equal to
$e^{\mathcal P(f_{|J(f)},\phi)}=e^{\mathcal P(f,\tilde\phi)}$.

\begin{defn}
Given a continuous function $\psi : J(f) \to \mathbb R$ and an invariant probability measure $m$ supported on $J(f)$, we define formally 
$${\rm Var}(\psi, m) \defeq \lim_{n \to \infty} \frac{1}{n} \int_{J(f)} \left| \sum_{i=0}^{n-1} \psi \circ f^i(x) \right|^2 dm
\in [0,+\infty].$$
\end{defn}

The following proposition is 
a direct consequence of 
{\bf (N4)} and {\bf (N5)}, together with the above extension lemmas;
see for instance
\cite{McMullen08,Parry90}. 

\begin{prop} \label{prop_deripressure}
Let $\{\phi_t\}_{t \in (-1,1)}$ be a smooth path in
$C^{\alpha}(J(f),\mathbb R)$ for some $\alpha \in (0,1)$. Suppose that
$\phi_0$ admits a unique equilibrium state
$m = m(\phi_0)$
and that $m$
has a 
strictly positive Lyapunov exponent.
Set $\dot{\phi_0} = d\phi_t/dt |_{t=0}$. Then we have \begin{equation}\label{eq:first-derivative}
\frac{d\mathcal{P}(\phi_t)}{dt} \bigg|_{t=0} = \int_{J(f)} \dot{\phi_0} dm.
\end{equation} If the  expressions
in \eqref{eq:first-derivative}
are equal to zero,
then we have 
\begin{equation*} 
\frac{d^2\mathcal{P}(\phi_t)}{dt^2} \bigg|_{t=0} = {\rm Var}(\dot{\phi_0}, m)  + \int \ddot{\phi_0}dm.
\end{equation*}
\end{prop}

In the statement above and in the rest of the paper, the notation $\dot{\phi_0}=d\phi_t/dt |_{t=0} $
stands for the \emph{function} which is defined pointwise as 
$\dot{\phi_0} (z)=d\phi_t(z)/dt |_{t=0}$
for every $z\in J(f)$. 

\begin{lem} \label{lem_variance_cobdry}
Fix $\phi,\psi \in C^{\alpha}(J(f))$.  Suppose that
$\phi$ admits a unique equilibrium state
$m = m(\phi)$
and that $m$
has
a strictly positive Lyapunov exponent.
Then we have
${\rm Var}(\psi, m(\phi)) = 0$ if and only if $\psi$ is $C^0$-cohomologous to 0, i.e., there exists a continuous function $h : J(f) \to \mathbb R$ such that $\psi = h - h\circ f$.
\end{lem}
\begin{proof}
It is a standard fact that we have
${\rm Var}(\psi, m)=0$ if and only if 
$\psi$ is a 
$L^2(m)$-coboundary,
i.e., if
there exists $h\in L^2 (m)$
such that 
$\psi = h-h\circ f$.
In order to conclude the proof, it is enough to show that
a $L^2$-coboundary is a $C^0$-coboundary. 
The proof is a standard application of the spectral gap for the norm $\|\cdot\|_\diamond$, see for instance \cite[Lemma 3.4 and Corollary 3.5]{FMT03}
and \cite[Proposition 5.9]{BD24eq2}.  
\end{proof}

Finally, let $W$ be an open subset of $\mathbb{C}^\ell$ for some $\ell \ge 1$. We say that 
a
map $W \ni w \mapsto \mathcal{L}_w \in L(B_\diamond)$ is {\it holomorphic} if for every $w \in W$, there exists $\mathcal{L}'_w \in L(B_\diamond)$ such that $||
h^{-1}  ( \mathcal{L}_{(w+h)}-\mathcal{L}_w )
-\mathcal{L}'_w||_\diamond \to 0$ as $h\to0$.
The following proposition gives a simple condition to ensure that this is the case 
for the transfer operators
associated to some family of functions $\zeta_\lambda$,
giving an improvement of 
{\bf (N4)}
when the parameter is complex and the dependence of the potential on the parameter is sufficiently regular.

\begin{prop}\label{p:SU-full}
Let $\{\zeta_w\}_{w \in W}$
be a family of functions in $B_\diamond$
such that
\begin{enumerate}
\item there exists $w_0 \in W$
such that
Lyapunov exponent of each equilibrium state of $\zeta_{w_0}$ is strictly positive, 

\item the function $w \mapsto \zeta_w$
is continuous with respect to $\|\cdot\|_\diamond$;
\item for every $z \in J(f)$,
the function $W \ni w \mapsto \zeta_w(z)$ is holomorphic.
\end{enumerate}
Then, the map $W \ni w \mapsto \mathcal{L}_{\zeta_w|_{J(f)}} \in L(B_\diamond (J))$ is holomorphic. 
\end{prop}

\begin{proof}
By the first assumption
and
Lemma \ref{lem_IRL},
we have 
$\mathcal L_{\zeta_{w_0}}\in L(B_\diamond)$. 
It is enough to show that the 
 map $W \ni w \mapsto \mathcal{L}_{\zeta_w} \in L(B_\diamond)$ is continuous
 as this, together with the third
 assumption, implies the holomorphicity of $\mathcal L_w$; see for instance \cite[Lemma 7.1]{UZ04real} or \cite[Lemma 5.1]{SU10}.

 In order to show the continuity of $w \mapsto \mathcal L_w$, by the 
 property {\bf (N3)} 
 of the norm $\|\cdot\|_\diamond$, 
 for every $w,w' \in W$
 we have
 \[
\|(\mathcal L_{w}-\mathcal L_{w'}) g\|_\diamond = \|f_*
\big((e^{\zeta_w} - e^{\zeta_{w'}}) g\big)\|_\diamond \lesssim 
\|
(e^{\zeta_w}- e^{\zeta_{w'}})
g\|_\diamond, 
 \]
where the implicit constant is independent of 
$w,w'$.
It follows from the above estimates
and {\bf (N2)}
that we have
\[
\|(\mathcal L_{w}-\mathcal L_{w'})  g\|_\diamond 
\lesssim 
\|e^{\zeta_w}- e^{\zeta_{w'}}\|_\diamond \|g\|_\diamond
\lesssim \|e^{\zeta_w}\|_\diamond
\|1- e^{\zeta_{w'}-\zeta_w}\|_\diamond \|g\|_\diamond
\lesssim
e^{\|\zeta_w\|_\diamond}
\|1- e^{\zeta_{w'}-\zeta_w}\|_\diamond \|g\|_\diamond,
\]
where again the implicit constants are independent of $w, w'$.
The assertion follows from the fact that
$e^{t}-1\lesssim t$ for small $t>0$,
{\bf (N2)},
and the assumption on the continuity of $\zeta_w$ with respect to $\|\cdot\|_\diamond$.
\end{proof}

\subsection{Applications to $\phi = -\theta \log |f'|$} \label{sec_app_log_f'}
When $f$ only has attracting or parabolic periodic points, the function $\phi =- \log |f'|$ is smooth on $J(f)$. Up to redefining it outside of some neighbourhood of $J(f)$,
the results of the previous section apply, up to multiplying
it by a sufficiently small constant (in order to verify the requirement $\max \phi < \mathcal P(f,\phi)$).
We then have the following lemmas, which we will refer to several times in the sequel. They are both well-known (see for example \cite{DU91}), but we give the proofs to show how they can be deduced from the formalism of the previous section.

\begin{lem}\label{lem_Pressure>0}
Let $\theta>0$ be such that  $\mathcal{P}(-\theta\log |f'|)>0$. Then every equilibrium state of
the function $-\theta\log |f'|: J(f) \to \mathbb R$ has strictly positive Lyapunov exponent.
\end{lem}
\begin{proof}
We have
 \begin{equation} \label{eq_pressure_positive}
0< \mathcal{P}(-\theta\log|f'|) = h_\nu - \theta L_\nu,
\end{equation}
where $\nu$ is an
equilibrium state of $-\theta\log|f'|$
and $h_\nu$ and $L_\nu$ are  the measure-theoretic entropy and the Lyapunov exponent of $\nu$, respectively. 
Since $\nu$ is supported on the Julia set $J$, its Lyapunov exponent is non-negative, i.e., we have
$L_\nu \ge 0$
\cite{P93}.
Hence,
as $\theta>0$, \eqref{eq_pressure_positive}
implies that $h_\nu>0$. Since the Hausdorff dimension of $\nu$ is non-negative
and satisfies ${\rm H.dim}(\nu) \cdot L_\nu = h_\nu$, this gives
$L_\nu>0$.     
\end{proof}

\begin{lem}\label{lem_pressure_decreasing}
The function $\theta \mapsto \mathcal{P}(-\theta \log |f'|)$ is strictly decreasing as $\theta$ increases from 0 to the first zero of the function $\mathcal{P}(-\theta\log |f'|)$.
\end{lem}
\begin{proof}
Take $\theta_0\in [0,T)$, where
$T>0$ is 
the first zero of the function $\theta \mapsto \mathcal{P}(-\theta\log |f'|)$.
We apply Proposition \ref{prop_deripressure} with $\phi_s = -(\theta_0+s)\log|f'|$, 
for $s$ in a sufficiently small neighbrouhood of $0$. Observe that
$\phi_0 = -\theta_0\log |f'|$.
Since $\mathcal{P}(\phi_0)>0$, by Lemmas \ref{lem_Pressure>0}
and \ref{lem_IRL}
$\phi_0$ has a unique equilibrium state $\nu=\nu(\phi_0)$, whose Lyapunov exponent
$L_\nu$
is strictly positive.

By \eqref{eq:first-derivative}, we have
\begin{equation*}
\frac{d\mathcal{P}(-\theta\log |f'|)}{dt} \bigg|_{\theta=\theta_0} =
\frac{d\mathcal{P}(\phi_s)}{ds} \bigg|_{s=0} = -\int_{J(f)} \log |f'| d\nu = -L_\nu<0.
\end{equation*} 
The conclusion follows.
\end{proof}

By 
\cite{StratU}, 
the pressure function $t \mapsto \mathcal{P}(-t\log|f'|)$ is actually
analytic on the interval between $0$ and its first zero. We will reprove this fact in Section \ref{sec_AnalyticityBowenParameter}, where we will show a stronger
joint analyticity property
for the pressure function in both $t$ and $f$.

\subsection{Extension of analytic functions}\label{ss:extension-analytic}
In later sections, we will
often need to extend a family of transfer operators, depending analytically on some real parameters, to a {\it holomorphic} family of transfer operators, in order to apply the perturbation theory on the spectra of the operators. To achieve this, we will use
Lemma \ref{lem_6.3} below.
We also
refer to \cite{Rugh08dimension} for a different method of extension.

\medskip

For every integer $d \ge 1$, consider the embedding $\iota_d \colon \mathbb{C}^d \to \mathbb{C}^{2d}$ given by 
\begin{equation}\label{eq_iota}
(x_1+iy_1,\ldots,x_d+iy_d) \mapsto (x_1,y_1,\ldots,x_d,y_d).
\end{equation}
Observe, in particular,
that $\mathbb C^d$ is embedded by $\iota_d$ in $\mathbb C^{2d}$ as the set of points of real coordinates
(which, in turn, is 
parametrized by $\mathbb C^d$ by means of $\iota_d$). 
For every $z \in \mathbb{C}^\ell$
and every $r>0$, denote by $D_\ell(z,r)$ 
the 
$\ell$-dimensional 
polydisk in $\mathbb{C}^\ell$
centered at $z$ and with radius $r$.
Observe that, with the above identification, we 
have $\iota_d (D_d (0,r))\subset D_{2d} (0,r)$ and, more generally,
$\iota_d (D_d (z,r))\subset D_{2d} (\iota_d(z),r)$
for every $z \in \mathbb C^d$. 

\begin{lem}[{\cite[Lemma 6.4]{SU10}}] \label{lem_6.3}
For every $M \ge 0$, 
$R>0$, 
 $\lambda_0 \in \bbC^d$, and 
 every complex analytic function
$\phi: D_d(\lambda_0,R) \to \bbC$
which is bounded in modulus by $M$, there exists a complex analytic function $\tilde{\phi} : D_{2d}(\iota_d (\lambda_0),R/4) \to \bbC$ that is bounded in modulus by $4^dM$ and such that the restriction of $\tilde \phi \circ \iota_d$ 
to
$D_d(\lambda_0,R/4)$ 
coincides with the real part $\Re(\phi)$ of $\phi$.
\end{lem}

\section{Analyticity of the Bowen function} \label{sec_AnalyticityBowenParameter}
We fix in this section a parabolic subfamily $\Lambda$ of ${\rm rat}^{cm}_D$. Recall that this implies that
a 
(possibly empty) subset of the critical points is persistently contained in parabolic basins. In particular, these critical points are passive on all of $\Lambda$. 
We also fix a $\Lambda$-hyperbolic
component $\Omega$ in $\Lambda$.
Recall
that we denote by $J_\lambda$
the Julia set of $f_\lambda$.
Observe also that, for every $\lambda \in \Omega$, the map
$\log|f'_\lambda|$ is
 H\"older continuous in a neighbourhood of
$J_\lambda$.

\medskip

Fix $\eta \in (0, \log D)$. Given $\lambda\in\Omega$, define $\delta_\eta(\lambda)$ to be the
unique
real
number such that
$$\mathcal{P}(-\delta_\eta(\lambda)\log|f'_\lambda|) = \eta.$$
Note that the number $\delta_\eta(\lambda)$ exists and is unique as
the pressure function $t \mapsto \mathcal{P}(-t\log|f'_\lambda|)$ is strictly decreasing
as $t$ increases from $0$ up to the first zero of $\mathcal{P}$ by Lemma \ref{lem_pressure_decreasing}. Moreover, we have $0< \delta_\eta(\lambda) \leq 2$ for every $0<\eta<\log D$ and $\lambda\in\Omega$.
 If $\Lambda$ and $\Omega$ are such that $f_\lambda$ is hyperbolic for every $\lambda \in \Omega$
 and we take $\eta=0$,
 the number $\delta_0(\lambda)$ is called the {\it Bowen parameter}, as in this case,
Bowen proved that $\delta_0(\lambda)$ equals the Hausdorff dimension of the Julia set of $f_\lambda$; see \cite{Bowen79}.
In our case, abusing notation, we 
call $\delta_\eta(\lambda)$ the Bowen
 number
 of $f_\lambda$.
The main result of this section is
the following proposition, which generalizes Ruelle's result \cite{Ruelle82} to every
$\Lambda$-hyperbolic component.

\begin{prop} \label{prop_deltaanalytic}
The 
Bowen 
function
$\delta_{\eta}: \Omega \to \mathbb{R}$ given by $\lambda \mapsto \delta_\eta(\lambda)$ is real-analytic.
\end{prop}

\medskip

Fix $\lambda_0 \in \Omega$ and denote 
$\ell\defeq {\rm dim}_{\mathbb C} \Omega$.
In the following, we will work in a chart of $\Omega$ centred at $\lambda_0$. In particular, we can assume that $\lambda_0=0$ and $\lambda \in D_\ell (0,1)$.
Let $R_0<1$ be a real number such that the 
map $\Psi_\lambda:J_0\to J_\lambda$
conjugating the dynamical systems $(J_{0},f_{0})$ and $(J_\lambda,f_\lambda)$
is well-defined for all $\lambda\in D_\ell (0,R_0)$.
For $\lambda \in D_\ell (0,R_0)$ and  $\theta \in \bbR$, consider the potential function $\phi_{(\theta,\lambda)} : J_{0} \to \mathbb{R}$ given by
\begin{equation*}
    \phi_{(\theta,\lambda)}(z) \defeq -\theta \log \left| f'_{\lambda} \circ \Psi_\lambda (z) \right|.
\end{equation*}

Proposition \ref{prop_deltaanalytic} follows from the following proposition.
Indeed, once we establish that the function $(\theta,\lambda) \mapsto \mathcal{P}(\phi_{(\theta,\lambda)})$ is real-analytic in a neighbourhood of $(\delta_\eta (0),0)$,
since $\{\eta\}$ is a closed set in $\mathbb{R}$ and we have $\frac{d}{d\theta}\big|_{\theta= \delta_\eta(\lambda)}\mathcal{P}(-\theta\log|f'_\lambda|)<0$ for every $\lambda$
by Lemma \ref{lem_pressure_decreasing}, 
it follows from the analytic
implicit function theorem that
the function 
$\lambda \mapsto \delta_\eta(\lambda)$
is analytic on $\Omega$.

\begin{prop} \label{prop_analytic_Mis}
There exists $R>0$ such that the function $(\theta,\lambda) \mapsto \mathcal{P}(\phi_{(\theta,\lambda)})$
 is real-analytic 
 for 
 $ (\theta,\lambda) \in (\delta_\eta(0)-R,\delta_\eta(0)+R)\times D_\ell (0, R)$.
\end{prop}

The rest of this section is devoted to the proof of Proposition \ref{prop_analytic_Mis}. Our proof follows the same general outline as that of \cite[Proof of Theorem A]{SU10} or \cite[Proof of Theorem 9.3]{UZ04real}.
The spectral gap property for the transfer operators $\mathcal L_{\phi(\theta,\lambda)}$
with respect to the norm $\|\cdot\|_\diamond$
as in Section \ref{ss:operators} will allow us to 
adapt those arguments in 
our context.

\medskip

For every $z \in J_{0}$, consider the map
$$\psi_z(\lambda) \defeq \frac{f'_\lambda \circ \Psi_\lambda(z)}{f'_{0}(z)}.$$
It follows from the absolute continuity of the holomorphic motion of the Julia sets
that,
shrinking $R_0$ if necessary, for all $z \in J_{0}$ 
and $\lambda\in D_\ell (0,R_0)$, 
we have
$$|\psi_z(\lambda)-1|<1/5.$$
Then,
for every $z \in J_{0}$, there exists a branch of $\log\psi_z$ 
sending 
$0$ to $0$ and whose modulus is bounded by $1/4$.
Applying Lemma \ref{lem_6.3} to the complex analytic function $\log \psi_z$, we see that the real analytic function
 $\Re\log \psi_z : D_\ell (0,R_0) \to \mathbb{R}$ has an analytic extension $\Re \widetilde{\log \psi_z} : D_{2\ell} (0,R') \to \mathbb{R}$ for some $R' \in (0,R_0)$. 
Recall that $D_\ell (0,R_0)$ is seen
as a subset of the points
of  
$\mathbb C^{2\ell}$
with real coordinates
by means of the immersion $\iota_\ell$ as in \eqref{eq_iota}, and that we have
$\iota_\ell (0)=0\in \mathbb C^{2\ell}$.

\medskip

For $(\theta,\lambda) \in \mathbb{C} \times  D_{2\ell}(0,R')$,
define $\zeta_{(\theta,\lambda)}:J_{0} \to \mathbb C$ by
$$\zeta_{(\theta,\lambda)}(z) \defeq -\theta \Re \widetilde{\log \psi_z}(\lambda) + \theta \log|f'_{0}(z)|.$$
Note that $\zeta_{(\theta,\lambda)}=
\phi_{(\theta,\lambda)}$ for
every
$(\theta,\lambda) \in \mathbb{R} \times D_\ell (0,R')$.

\medskip

Let $\beta$ be such that the conjugacy map $\Psi_\lambda : J_{0} \to J_\lambda$ is $\beta$-H\"older continuous for all $\lambda \in D_\ell (0,R')$. We first note that the map 
$(\theta,\lambda) \mapsto \zeta_{(\theta,\lambda)}$ is continuous with respect to the $\beta$-H\"older norm
(see for instance \cite[Lemma 6.6]{SU10} for a similar computation).
Applying
Lemma \ref{l:extension} to the  families $\Re \zeta_{(\theta,\lambda)}$ 
and
$\Im \zeta_{(\theta,\lambda)}$, and
shrinking $R'$, gives an extended family
$\mathbb{C}\times D_{2\ell} (0,R') \ni (\theta,\lambda) \mapsto \widetilde{\zeta}_{(\theta,\lambda)} : \mathbb P^1 \to \mathbb C$ which is
continuous with respect to the $\beta/2$-H\"older norm and satisfies
$ \widetilde{\zeta}_{(\theta,\lambda)} (z)=
{\zeta}_{(\theta,\lambda)} (z)$
for every 
$(\theta,\lambda) \in \mathbb{C} \times \iota_1^{-1} D_{2\ell}(0,R')$ and $z\in J_0$. 

\medskip

Fix a norm 
$||\cdot||_\diamond$ 
as in Section \ref{ss:operators}
such that $||\cdot||_\diamond \lesssim ||\cdot||_{C^{\beta/2}}$ 
and recall that we denote by
$B_\diamond (J_0)$
the
Banach space defined by the norm $||\cdot||_\diamond$
as in \eqref{eq:def-BJ}, and by $L(B_\diamond (J_0))$ the space of the bounded
linear operators on $(B_\diamond (J_0), \|\cdot\|_\diamond)$.
For $(\theta,\lambda) \in \mathbb{C}\times D_{2\ell} (0,R')$, consider the complex transfer operator $\mathcal{L}_{(\theta,\lambda)} \defeq \mathcal{L}_{\widetilde \zeta_{(\theta,\lambda)}}$.

\begin{lem} \label{lem_op_hol}
There exists $R'' \in (0,R')$ such that the map 
$D_\ell (\delta_\eta(0) , R'') \times  D_{2\ell} (0,R'')
\ni (\theta,\lambda) \mapsto \mathcal{L}_{(\theta,\lambda)}\in L(B_\diamond(J))$ is holomorphic. 
\end{lem}

In the proof of the above lemma, 
we will use
the fact that we have
$\mathcal{P}(\phi_{(\delta_\eta(\lambda),\lambda)}) \equiv \eta>0$ to verify the assumptions of
the lemmas in Section \ref{sec_app_log_f'}, and hence of Proposition \ref{p:SU-full}.
This is the reason why 
we take 
$\eta \in (0,\log D)$, rather than $\eta = 0$
as in Bowen's paper \cite{Bowen79}.

\begin{proof}
We 
verify that the three 
assumptions of Proposition \ref{p:SU-full} are satisfied 
when considering
the family of weights $\tilde \zeta_{(\theta, \lambda)}$
instead of $\zeta_w$.

\smallskip

For (1), 
since we have
$\mathcal{P}(\phi_{(\delta_\eta(0),0)}) = \eta > 0$,
Lemma \ref{lem_Pressure>0} ensures that
the
desired property
holds
at the parameter $w_0 = (\delta_\eta(0), 0)$.

\smallskip

For (2),
we observe that, up 
to choosing  $R'' \in (0,R')$
sufficiently small,
the family
$(\theta,\lambda) \mapsto \widetilde \zeta_{(\theta,\lambda)}$ is continuous with respect to the $\beta/2$-H\"older norm by Lemma \ref{l:extension}. 
The desired condition then holds  
thanks to 
the bound $||\cdot||_\diamond \lesssim ||\cdot||_{C^{\beta/2}}$ which we required in the choice of $\|\cdot\|_\diamond$.

\smallskip

For (3),
for each fixed $z \in J_{0}$, the map $(\theta,\lambda) \mapsto \zeta_{(\theta,\lambda)}(z)$ is holomorphic by definition.

\smallskip

As all the assumptions of Proposition \ref{p:SU-full} are satisfied,
the assertion
follows
from 
Proposition \ref{p:SU-full}.
\end{proof}

We can now conclude the proof of 
Proposition \ref{prop_analytic_Mis}, which also concludes the proof of Proposition \ref{prop_deltaanalytic}.

\begin{proof}[End of the proof of
Proposition \ref{prop_analytic_Mis}] \label{proof_end32}

We let $R''$ be as in Lemma \ref{lem_op_hol}
and denote by $ 
\mathcal{L}_{\zeta_{(\theta,\lambda)}} \in L(B_\diamond(J))$ 
the transfer operator associated to $(\theta, \lambda)\in D_\ell (\delta_\eta(0),R'')\times D_\ell (0,R'')$.
For every $(\theta,\lambda) \in \mathbb R \times D_\ell (\lambda_0,R'')$, the value
$e^{P(\zeta_{(\theta,\lambda)})}$
is a simple isolated eigenvalue of the transfer operator $\mathcal{L}_{\zeta_{(\theta,\lambda)}}$.
It follows from 
Lemma \ref{lem_op_hol}
and
the Kato-Rellich perturbation theorem \cite{Kato95}
that 
there exists $0<R<R''$ and an analytic function $\mathcal{A}: D_\ell (\delta_\eta(0),R) \times D_{2\ell}(0,R) \to \mathbb C$ such that
$\mathcal{A}(\theta,\lambda)$ is a simple isolated eigenvalue for $\mathcal{L}_{(\theta,\lambda)}$
for every $(\theta,\lambda) \in D_\ell (\delta_\eta(0),R) \times D_{2\ell}(0,R)$
and $e^{P(\zeta_{(\theta,\lambda)})} = \mathcal{A}(\theta,\lambda)$ for every $(\theta, \lambda) \in (\delta_\eta (0)- R, \delta_\eta(0)+R) \times D_{2\ell} (0,R)$. 
The assertion follows by possibly reducing $R$,
taking the logarithm of 
the function $\mathcal A$, and recalling that $D_\ell(0,R)$ 
is identified 
to a subset of $\mathbb C^{2\ell}$ by means of the map $\iota_\ell$.
The proof is complete.
\end{proof}

\section{The Hessian
form $\langle \cdot, \cdot \rangle_G$ and the pressure form $\langle \cdot, \cdot \rangle_{\mathcal{P}}$} \label{sec_2form}
We continue to use the notations from the previous section. We fix a $\Lambda$-hyperbolic component $\Omega$ of 
a parabolic family $ \Lambda \subset {\rm rat}^{cm}_D$.
In this section, we show that the construction in \cite{HeNie23} of a positive semi-definite 
symmetric bilinear form $\langle \cdot, \cdot \rangle_G$ for hyperbolic components in the moduli space of rational maps can be extended to $\Omega$. To this end, we first construct a positive semi-definite symmetric bilinear form $\langle \cdot, \cdot \rangle_G$ on $\widetilde{\Omega}$, which is a lift of $\Omega$ in the parameter space ${\rm Rat}_D^{cm}$ and show that 
it descends to a positive semi-definite symmetric bilinear form on $\Omega$; see Sections \ref{sec:2-form-1} and \ref{sec:2-form-2}. We construct the pressure form
$\langle \cdot, \cdot \rangle_{\mathcal{P}}$ on $\Omega$ in Section \ref{ss:P-omega} and prove the conformal equivalence of $\langle \cdot, \cdot \rangle_G$ and $\langle \cdot, \cdot \rangle_{\mathcal{P}}$ in Section \ref{sec_conformal_equiv}.

\subsection{The Hessian
$2$-form
$\langle \cdot, \cdot \rangle_G$ on $\widetilde{\Omega}$}\label{sec:2-form-1}
Recall 
that  
$J_\lambda$ denotes the Julia set of $f_\lambda$,  for
 $\lambda \in \widetilde{\Omega}$.
Fix $\lambda_0 \in \widetilde{\Omega}$. Let $U(\lambda_0)$ be a neighborhood of $\lambda_0$
such that the conjugacy $\Psi_\lambda : J_{\lambda_0} \to J_\lambda$
is well-defined for any $f_\lambda \in U(\lambda_0)$.
Note that we can take 
$U(\lambda_0) = \widetilde{\Omega}$ if $\widetilde{\Omega}$ is simply connected.
Fix $\eta \in (0, \log D)$.
For simplicity, we denote by $\nu \defeq \nu_{\eta,\lambda_0}$
the unique
equilibrium state on $J_{\lambda_0}$ for the potential
$-\delta_{\eta}(\lambda_0)\log |f'_{\lambda_0}|$. 
We note that $-\delta_{\eta}(\lambda_0)\log |f'_{\lambda_0}|$ has a unique equilibrium state by the fact that $\mathcal{P}(-\delta_{\eta}(\lambda_0)\log |f'_{\lambda_0}|) = \eta > 0$, see 
Lemmas \ref{lem_Pressure>0} and \ref{lem_IRL}.

\begin{defn} \label{def_lyap}
The function $\mathrm{Ly}_{\lambda_0} : U(\lambda_0)\to \bbR$ is defined as
$$\mathrm{Ly}_{\lambda_0}(\lambda) \defeq  \int_{J_\lambda} \log|f'_\lambda| d\left((\Psi_\lambda)_* \nu\right)=\int_{J_{\lambda_0}} \log|f'_\lambda \circ \Psi_\lambda|d\nu.$$  
\end{defn}

\begin{lem} \label{lem_ly_harmonic}
The function $\mathrm{Ly}_{\lambda_0} : U(\lambda_0)\to \bbR$ is harmonic. In particular, it is real-analytic.
\end{lem} 
\begin{proof}
We 
show that $dd^c_\lambda \mathrm{Ly}_{\lambda_0} (\lambda)\equiv 0$, where $dd^c_\lambda$ denotes the complex Laplacian. 
Denote by $\pi_\Omega$ and $\pi_{\mathbb P^1}$ the natural projection of $\Omega \times \mathbb P^1$ onto its factors. With these notations, we see that, formally, we have
\begin{equation}\label{e:ddc-L}
dd^c_\lambda \mathrm{Ly}_{\lambda_0}
(\lambda)
=
(\pi_\Omega)_* ( dd^c_{\lambda,z} \log |f'_\lambda \circ \Psi_\lambda| \wedge (\pi_{\mathbb P^1})^* \nu).
\end{equation}
For every $z_0\in J_{\lambda_0}$, we have
\[
dd^c_{\lambda,z} \log |f'_\lambda \circ \Psi_\lambda (z_0)|=0
\]
since the map $\lambda \mapsto f'_\lambda \circ \Psi_\lambda (z_0)$ is a non-vanishing holomorphic function on $\Omega$ for every $z_0\in J_{\lambda_0}$. The assertion follows from Fubini's theorem, which also justifies that the currents and the equality
in \eqref{e:ddc-L} are well-defined.
\end{proof}

\begin{defn} \label{def_G_f}
The function $G_{\lambda_0}: U(\lambda_0) \to \bbR$ is defined as
$$G_{\lambda_0}(\lambda) \defeq \delta_{\eta}(\lambda)\mathrm{Ly}_{\lambda_0}(\lambda) = \delta_{\eta}(\lambda) \int_{J_{\lambda_0}} \log|f'_\lambda \circ \Psi_\lambda|d\nu.$$
\end{defn}

\begin{lem}
The function $G_{\lambda_0}: U(\lambda_0) \to \bbR$ is real-analytic. 
\end{lem}
\begin{proof}
The statement follows from Proposition \ref{prop_deltaanalytic} and Lemma \ref{lem_ly_harmonic}.
\end{proof}

Let $G:X \to \mathbb R$ be a smooth real-valued function on an $\ell$-dimensional smooth manifold $X$. The {\it Hessian} of $G$ at $x\in X$ with respect to a chart $u : U(x) \to \mathbb R^{\ell}$ is the map
$G_{(u)}''(x) : T_{x}X \times T_{x}X \to \mathbb R$ represented by the $2$-form
$$G_{(u)}''(x) = \sum_{i,j=1}^\ell \frac{\partial^2 G_{(u)}}{\partial u_i\partial u_j}(x) du_i \otimes du_j.$$
We note that if $u : U(x) \to \mathbb R^{\ell}$ and $v : V(x) \to \mathbb R^{\ell}$ are two different charts, we have $G''_{(u)}(x) = G''_{(v)}(x)$ as soon as 
$DG(x)$ vanishes, where $DG(x)$ is the differential $DG(x):T_xX \to T_{G(x)}\mathbb R$ (which 
a priori 
depends on the coordinate charts, but the condition of its vanishing does not);
see \cite[Section 7]{Bridgeman08}.
If $x\in X$ is such that $DG(x)=0$, then $G''(x) \defeq G''_{(u)}(x)$
(with respect to any chart $u$)
is a well-defined symmetric bilinear form $G''(x):T_{x}X \times T_{x}X \to \mathbb R$ on $T_xX$.
Moreover,
if $x$ is a local minimum for $G$, then $G''(x)$ is positive semi-definite. 

\medskip

In what follows, we will show that $G_{\lambda_0}$ has a minimum at $\lambda_0$. Therefore we have $DG_{\lambda_0}(\lambda_0)=0$ and the Hessian $G_{\lambda_0}''(\lambda_0) : T_{\lambda_0}\widetilde{\Omega} \times T_{\lambda_0}\widetilde{\Omega} \to \mathbb R$ is well-defined at $\lambda_0$
and is
 positive semi-definite.

\begin{prop}\label{prop_metric}
Fix $\lambda_0 \in\widetilde{\Omega}$.
We have 
$G_{\lambda_0}(\lambda_0) \le G_{\lambda_0}(\lambda)$ for all $\lambda \in U(\lambda_0)$.
\end{prop}

\begin{proof}
Recall that we denote by $\nu$ the unique
equilibrium state on $J_{\lambda_0}$ for the weight
$-\delta_{\eta}(\lambda_0)\log |f'_{\lambda_0}|$.
To ease notation, we also 
set $\nu_{\lambda}\defeq (\Psi_{\lambda})_\ast \nu$
for every 
$\lambda \in U(\lambda_0)$. Then we have
$$\mathrm{Ly}(\nu,f_\lambda)  = \int_{J_\lambda} \log|f'_\lambda|d \nu_{\lambda}
\quad
\mbox{ for all } \lambda \in U(\lambda_0).$$
	
Since 
the weight $-\delta_\eta(\lambda_0) \log|f'_{\lambda_0}|$
has pressure $\eta$
with respect to $f_{\lambda_0}$
and $\nu = \nu_{\lambda_0}$ is its (unique)
equilibrium state, by the definition of pressure
we have
\begin{equation} \label{eq_451}
h_{\nu_{\lambda_0}}(f_{\lambda_0}) = \eta + \delta_\eta(\lambda_0)\int_{J_{\lambda_0}} \log|f'_{\lambda_0}| d\nu_{\lambda_0}.
\end{equation}

Since  $\nu=\nu_{\lambda_0}$
is $f_{\lambda_0}$-invariant,  $\nu_\lambda$ is $f_\lambda$-invariant for every $\lambda \in U(\lambda_0)$.
Since the measure-theoretic
entropy is invariant under topological conjugacy, it follows that we have
\begin{equation} \label{eq_452}
h_{\nu_{\lambda_0}}(f_{\lambda_0}) =h_{(\Psi_\lambda)_\ast \nu_{\lambda_0}}(f_\lambda)=h_{\nu_\lambda}(f_\lambda).
\end{equation}
Again by the definition of pressure and the identity $P(-\delta_\eta(\lambda)\log|f'_\lambda|) = \eta$, we have
\begin{equation} \label{eq_453}
h_{\nu_\lambda}(f_\lambda) \le \eta + \delta_\eta(\lambda) \int_{J_\lambda} \log|f'_\lambda| d \nu_\lambda
\quad
\mbox{ for all } \lambda \in U(\lambda_0).
\end{equation}
Combining
\eqref{eq_451}, \eqref{eq_452},
and \eqref{eq_453}, we obtain the inequality
$\delta_\eta(\lambda_0)\mathrm{Ly}(\nu,f_{\lambda_0})\le\delta_\eta(\lambda)\mathrm{Ly}(\nu,f_\lambda)$. The assertion follows.
\end{proof}

Therefore,
the Hessian of $G_{\lambda_0}$ at $\lambda_0$ defines a positive semi-definite symmetric bilinear form $\langle \cdot, \cdot \rangle_G$ on the tangent space $T_{\lambda_0}\widetilde{\Omega}$ as follows. 

\begin{defn} \label{def_G_parameter}
For every
$\vec{u},\vec{v} \in T_{\lambda_0}\widetilde{\Omega}$,
we define 
$$\langle \vec{u},\vec{v} \rangle_G \defeq (G_{\lambda_0}''(\lambda_0))(\vec{u},\vec{v}).$$
\end{defn}

For every $\vec{v}
\in T_{\lambda_0}\widetilde{\Omega}$,
we will also denote
$\|\vec{v}\|_G \defeq \sqrt{\langle \vec{v}, \vec{v}\rangle_G}$.

\begin{lem} \label{rmk_47}
Let $\gamma(t), t \in (-1, 1)$
be a smooth
path in $U(\lambda_0)$ with $\gamma(0) = \lambda_0$ and $\gamma'(0)=\vec{v} \in T_{\lambda_0}\widetilde{\Omega}$. Then 
$$||\vec{w}||_G^2 = \left.\frac{d^2}{dt^2} \right|_{t=0} G_{\lambda_0}(\gamma(t)).$$
\end{lem}
\begin{proof}
For every $\gamma$ as in the statement,
we have

\[
\left.\frac{d^2}{dt^2} \right|_{t=0} G_{\lambda_0}(\gamma(t))
= G''_{\lambda_0}(\gamma(0))(\gamma'(0),\gamma'(0)) + DG_{\lambda_0}(\gamma(0)) \cdot \gamma''(0)\\
=G''_{\lambda_0}(\lambda_0)
(\vec{v},\vec{v}).
\]
The second equality follows from the identity $DG_{\lambda_0}(\gamma(0)) = DG_{\lambda_0}(\lambda_0) = 0$
given by Proposition \ref{prop_metric}. 
The assertion follows.
\end{proof}

\subsection{The Hessian
$2$-form
$\langle \cdot, \cdot \rangle_G$
on $\Omega$}\label{sec:2-form-2}  
We now show 
that the Hessian form $\langle \cdot, \cdot \rangle_G$
on (the tangent bundle of)
$\widetilde \Omega$
descends to a symmetric bilinear form on
(the tangent bundle of)
$\Omega$. 
We first need a preliminary lemma.
\begin{lem} \label{lem_mobius_conj_coho}
Fix
$\lambda_0 \in \widetilde \Omega$. If $f_1,f_2 \in \widetilde{\Omega}$ are M\"obius conjugate, then $\log|f'_1
\circ 
\Psi_{f_1}|$
and
$
\log|f'_2
\circ 
\Psi_{f_2}|$
are $C^0$-cohomologous. In particular, we have
$\delta_{\eta}(\lambda_{f_1})=\delta_{\eta}(\lambda_{f_2})$. Here $\Psi_{f_i} : J_{\lambda_0} \to J(f_i), i=1,2$ is the conjugacy map.
\end{lem}
\begin{proof}
Let $g$ be a 
M\"obius transformation such that
$f_1 = g\circ f_2\circ g^{-1}$
and observe that
$g= \Psi_{f_1} \circ \Psi_{f_2}^{-1}$.
Then, for every $x \in J_{\lambda_0}$, we have 
\begin{align*}
\log|f'_1\circ \Psi_{f_1}(x)| &= \log|f'_1\circ g\circ \Psi_{f_2}(x)|=\log | 
(g\circ f_2\circ g^{-1})' (g\circ \Psi_{f_2}(x))|\\
&= \log |g' (f_2\circ \Psi_{f_2}(x))|
+ \log |f'_2 (\Psi_{f_2}(x))| + \log |(g^{-1})' (g\circ \Psi_{f_2}(x))|\\
&=\log | g' (f_2\circ \Psi_{f_2}(x))|+ \log |f'_2 (\Psi_{f_2}(x))| - 
\log |g' (\Psi_{f_2}(x))|\\
&=\log | g' (\Psi_{f_2}\circ f_{\lambda_0}(x))|+ \log |f'_2 (\Psi_{f_2}(x))| - 
\log |g' (\Psi_{f_2}(x))|.
\end{align*}
The first 
assertion follows by
setting $h\defeq
\log |g' (\Psi_{f_2}(x))|$
and observing that $h$ is continuous
as $g$ has no critical points.
As the pressure function
only depends on cohomology classes \cite{Parry90}, 
we have
$\mathcal{P}(-t\log|f'_1
\circ 
\Psi_{f_1}|) = \mathcal{P}(-t\log|f'_2
\circ 
\Psi_{f_2}|)$ for all $t \in \mathbb R$. Therefore,
we have $\delta_{\eta}(f_1)=\delta_{\eta}(f_2)$ and the proof is complete. 
\end{proof}

\begin{lem}
Fix $[\lambda] \in\Omega$ and $\vec{v}\in T_{[\lambda]}\Omega$.
Let $\gamma(t), t\in(-1,1)$ be a smooth
curve in $\Omega$ with $\gamma(0)= [\lambda]$ and $\gamma'(0)=\vec{v}$.
If $\widetilde{\gamma}_1(t)$ and $\widetilde{\gamma}_2(t)$ are two smooth
lifts of $\gamma(t)$ to $\widetilde{\Omega}$, then we have
$G_{\widetilde{\gamma}_1(0)}(\widetilde{\gamma}_1(t))=G_{\widetilde{\gamma}_2(0)}(\widetilde{\gamma}_2(t))$ for every $t \in (-1, 1)$.
\end{lem}
\begin{proof}
Let 
$\widetilde{\gamma}_1(t)$ and $\widetilde{\gamma}_2(t)$ 
be two
lifts of $\gamma(t)$ to 
$\widetilde{\Omega}$
as in the statement. 
Since the definitions are local, we may assume that we have
${\widetilde{\gamma}_1(t)}\subset U({\widetilde{\gamma}_1(0)})$ and ${\widetilde{\gamma}_2(t)}\subset U({\widetilde{\gamma}_2(0)})$
for all $t \in (-1,1)$,
where the neighbourhood $U({\lambda})$ 
of $\lambda \in \widetilde \Omega$ is as in the previous section.
Since $\widetilde{\gamma}_1(t)$ and $\widetilde{\gamma}_2(t)$ are M\"obius conjugate, we have $\delta_{\eta}(\widetilde{\gamma}_1(t)) = \delta_{\eta}(\widetilde{\gamma}_2(t))$ 
for every $t \in (-1,1)$
by Lemma \ref{lem_mobius_conj_coho}.
Let $M : J_{\widetilde{\gamma}_2(0)} \to J_{\widetilde{\gamma}_1(0)}$ be the M\"obius conjugacy map, and $\nu_i$ be the equilibrium state of $-\delta_\eta(\widetilde{\gamma}_i(0))\log |f_{\widetilde{\gamma}_i(0)}'|$ 
for $i=1,2$. 

\medskip

We first claim that $M_\ast \nu_2 = \nu_1$.
We compute
\begin{align*}
h_{M_\ast \nu_2}(f_{\widetilde{\gamma}_1(0)}) + \int_{J_{\widetilde{\gamma}_1(0)}} \log |f_{\widetilde{\gamma}_1(t)}'|d(M_\ast \nu_2) &=h_{\nu_2}(f_{\widetilde{\gamma}_2(0)}) + \int_{M^{-1}J_{\widetilde{\gamma}_1(0)}} \log |f_{\widetilde{\gamma}_1(t)}' \circ M|d\nu_2\\
&=h_{\nu_2}(f_{\widetilde{\gamma}_2(0)}) + \int_{J_{\widetilde{\gamma}_2(0)}} \log |f_{\widetilde{\gamma}_2(t)}'|d\nu_2 \\
&= \mathcal{P}(-\delta_\eta(\widetilde{\gamma}_2(0))\log |f_{\widetilde{\gamma}_2(0)}'|)
~(\text{as $\nu_2$ is an equilibrium state})\\
&= \mathcal{P}(-\delta_\eta(\widetilde{\gamma}_1(0))\log |f_{\widetilde{\gamma}_1(0)}'|)
~(\text{by Lemma \ref{lem_mobius_conj_coho}}).
\end{align*}
Therefore,
$M_\ast \nu_2$ is an equilibrium state for $-\delta_\eta(\widetilde{\gamma}_1(0))\log |f_{\widetilde{\gamma}_1(0)}'|$.
Since $-\delta_\eta(\widetilde{\gamma}_1(0))\log |f_{\widetilde{\gamma}_1(0)}'|$ has a unique equilibrium state, this gives the desired equality 
$M_\ast\nu_2 = \nu_1$.

\medskip

It follows from the above
that we have
\begin{align*}
G_{\widetilde{\gamma}_1(0)}(\widetilde{\gamma}_1(t))&=\delta_{\eta}(\widetilde{\gamma}_1(t)) \int_{J_{\widetilde{\gamma}_1(0)}} \log |f_{\widetilde{\gamma}_1(t)}' |d\nu_1 =\delta_{\eta}(\widetilde{\gamma}_2(t)) \int_{M(J_{\widetilde{\gamma}_2(0)})} \log |f_{\widetilde{\gamma}_1(t)}' |d (M_\ast \nu_2) \\
&=\delta_{\eta}(\widetilde{\gamma}_2(t)) \int_{J_{\widetilde{\gamma}_2(0)}} \log |f_{\widetilde{\gamma}_1(t)}'  \circ M|d \nu_2 
=G_{\widetilde{\gamma}_2(0)}(\widetilde{\gamma}_2(t))
\end{align*}
for every $t\in (-1,1)$. The assertion follows. 
\end{proof}

The above lemma implies that the following definition is well-posed.
\begin{defn}\label{def_metric_G}
For every
$[\lambda]\in\Omega$ and $\vec{v}\in T_{[\lambda]}\Omega$, we define 
$$||\vec{v}||_G \defeq \|\widetilde{\vec{v}}\|_G = ||\widetilde{\gamma}'(0)||_G$$
where
$\tilde \lambda$ is any representative of $[\lambda]$,
$\widetilde{\vec{v}}$ is a lift of the tangent vector $\vec{v}$ to $T_{\tilde \lambda}\tilde \Omega$,
and $\widetilde{\gamma}(t)$, 
$t\in(-1,1)$ is any
smooth real $1$-dimensional curve in $\widetilde\Omega$ with $\widetilde{\gamma}(0)= \widetilde\lambda$ and $\gamma'(0)=\widetilde{\vec{v}}$.
\end{defn}

\subsection{The pressure $2$-form 
$\langle \cdot, \cdot \rangle_{\mathcal{P}}$ on $\Omega$}\label{ss:P-omega}
From now on, we will use $\lambda$
(instead of $[\lambda]$ as before)
to denote an element of $\Omega$. By abuse of notation, we will refer to $f_\lambda$ and $J_\lambda$ as the objects corresponding to $\lambda$, even if a priori $\lambda$ should stand for an equivalence class. This identification is meaningful by the lemmas in the previous section, as we can choose a lift of representatives of the elements of $\Omega$ in $\widetilde \Omega$.

In this section, we construct the pressure form $\langle \cdot, \cdot \rangle_{\mathcal{P}}$ on $\Omega$. 
Fix $\lambda_0 \in \Omega$. Since,
for every $\lambda \in \Omega$, the dynamical system $(J_\lambda,f_\lambda)$ is topologically
conjugate to $(J_{\lambda_0},f_{\lambda_0})$
by means of the holomorphic motion,
we can think of $(J_{\lambda_0},f_{\lambda_0})$ as a {\it model dynamics} for
every $f_\lambda$ with $\lambda\in\Omega$. 
Recall that we denote by
$\Psi_\lambda: J_{\lambda_0} \to J_\lambda$  the conjugacy map. 

\medskip

For a fixed $\eta \in (0,\log D)$, recall that  $\delta_{\eta}(\lambda)$ is the unique real number 
satisfying
$$\mathcal{P}(-\delta_{\eta}(\lambda)\log|f'_\lambda\circ \Psi_\lambda|) = \eta.$$  

Let $\mathcal{C}_\eta(J_{\lambda_0})$ be the set of cohomology classes of H\"older continuous functions with pressure $\eta$
with respect to $f_{\lambda_0}$,
that is,
$$\mathcal{C}_\eta(J_{\lambda_0}) \defeq \{\phi: \phi \in C^{\alpha}(J_{\lambda_0},\mathbb{R}) \text{ for some } \alpha>0, \mathcal{P}(\phi) = \eta \}/ \sim$$
where $\phi_1 \sim \phi_2$ if $\phi_1$ and  $\phi_2$ are $C^0$-cohomologous
on $J_{\lambda_0}$.

\begin{defn}
The {\it thermodynamic mapping} $\mathscr{E} : \Omega \to\mathcal{C}_\eta(J_{\lambda_0})$ 
is defined 
by
$$\mathscr{E}(\lambda) \defeq [-\delta_\eta(\lambda)\log|f'_\lambda \circ \Psi_\lambda|].$$   
\end{defn}

The thermodynamic mapping is well-defined by Lemma \ref{lem_mobius_conj_coho}.

\medskip

Fix $\lambda \in \Omega$ and $\mathscr{E}(\lambda) \in \mathcal{C}_\eta(J_{\lambda_0})$.
Let $\nu$ be the equilibrium state for
a representative $\phi$ of 
$\mathscr E (\lambda)$,
whose existence and uniqueness 
are guaranteed by the fact that $\mathcal{P}(-\delta_\eta(\lambda)\log|f'_\lambda \circ \Psi_\lambda|) = \eta > 0$
and Lemmas \ref{lem_Pressure>0}
and 
\ref{lem_IRL}.
By Proposition \ref{prop_deripressure},
the tangent space 
of $\mathcal{C}_\eta(J_{\lambda_0})$ at $\mathscr{E}(\lambda)$ can be identified with 
$$T_{\mathscr{E}(\lambda)}\mathcal{C}_\eta(J_{\lambda_0}) = \left\{\psi 
\colon \psi
\in C^{\alpha}(J_{\lambda_0},\mathbb{R}) \text{ for some } \alpha>0,
\int_{J_{\lambda_0}} \psi d\nu = 0  \right\} / \sim,$$
where we used the fact
that, by definition, the pressure is constant on $\mathcal C_\eta(J_{\lambda_0})$.
Following \cite[p. 375]{McMullen08}, we define the {\it pressure form $||\cdot||_{pm}$ on $\mathscr{E}(\Omega) \subset \mathcal{C}_\eta(J_{\lambda_0})$}
as follows. 
Given $\psi \in T_{\mathscr{E}(\lambda)}\mathcal{C}_\eta(J_{\lambda_0})$,
we define
$$||[\psi]||^2_{pm} \defeq \frac{{\rm Var}(\psi,\nu)}{\eta -\int_{J_{\lambda_0}} \phi d\nu}.$$

Given $\vec{w} \in T_{\lambda}\Omega$, let $\gamma(t)=[f_t]$,
$t \in (-1,1)$ be a smooth path in $\Omega$
with $\gamma(0)=\lambda$ and $\gamma'(0)= \vec{w}$.
Letting $\phi_t$ be a representative of the class $\mathscr{E}(\gamma(t))$,
we define
\begin{equation} \label{def_pressure_form}
\|\vec{w}\|_{\mathcal P}
\defeq ||\dot{\phi}_0||_{pm} =
\frac{{\rm Var}(\dot{\phi}_0,\nu)}{\eta -\int_{J_{\lambda_0}} \phi_0 d\nu}.
\end{equation}

We
call $\|\cdot\|_{\mathcal P}$
the \emph{pressure form on $\Omega$}.
Observe that 
$||\cdot||_{\mathcal{P}}$ is positive semi-definite on $T_\lambda\Omega$ since 
we have ${\rm Var}(\dot{\phi}_0,\nu) \ge 0$
and $\eta -\int_{J_{\lambda_0}} \phi_0 d\nu > 0$.

\subsection{Conformal equivalence of $\langle \cdot, \cdot \rangle_G$ and $\langle \cdot, \cdot \rangle_{\mathcal P}$ on $\Omega$}\label{sec_conformal_equiv}
We continue to denote by $\lambda$ the elements of $\Omega$, as explained at the beginning of 
the previous section.

Fix 
$\lambda_0 \in \Omega$ and $\vec{v} \in T_{\lambda_0}\Omega$. Let $\gamma(t)$,
$t\in (-1,1)$ be a smooth path in $\Omega$ with $\gamma(0) =\lambda_0$ and $\vec{v} = \gamma'(0)$. Let $\widetilde \gamma(t)$ be a lift of $\gamma(t)$ in $\widetilde{\Omega}$, and denote by $f_t$ a map corresponding to $\widetilde \gamma (t)$.

\medskip

For every $x \in J_{\lambda_0}$ and $t \in (-1, 1)$, set
$$g(t,x) \defeq -\delta_{\eta}(\widetilde\gamma(t))\log|f'_t\circ \Psi_{\widetilde{\gamma}(t)}(x)|.$$ 
Denote by $\dot{g} (0,\cdot)$ and 
$\ddot{g}(0,\cdot)$
 the real-valued functions on $J_{\lambda_0}$ given by
 \[\dot{g}(0,x) = \frac{d}{dt}\Big|_{t=0} g(t,x)
 \quad \mbox{ and }
 \quad
\ddot{g}(0,x) = \frac{d^2}{dt^2}\Big|_{t=0} g(t,x)
\quad
\mbox{ for all } x \in J_{\lambda_0}\]
and denote 
by
$\nu$
the
(unique)
equilibrium state 
of
$g(0,x)=-\delta_{\eta} (\lambda_0) \log |f'_{0}|$.

\begin{prop}\label{prop_conf}
We have
	$$||\vec{v}||_{\mathcal{P}}^2 = \frac{||\vec{v}||_G^2}{\eta -\int_{J_{\lambda_0}}g(0,x) d\nu(x)}.
 $$
In particular, the Hessian
form $||\cdot||_G$ is conformal equivalent
to the pressure form $||\cdot||_{\mathcal{P}}$. 
\end{prop}
\begin{proof}
By the definition of $g(t,x)$, $\vec{v}$ can be identified with
$\dot{g}(0,\cdot)$.
By \eqref{def_pressure_form} and
Proposition \ref{prop_deripressure} (2), we have
$$||\vec{v}||_\mathcal{P}^2 = \frac{{\rm Var}(\dot{g}(0,x), \nu)}{\eta -\int_{J_{\lambda_0}}g(0,x) d\nu(x)} = \frac{-\int_{J_{\lambda_0}} \ddot{g}(0,x) d\nu(x)}{\eta - \int_{J_{\lambda_0}}g(0,x) d\nu(x)}.$$
Hence, it is enough to show that we have
\[
||\vec{v}||_G^2 = -\int_{J_{\lambda_0}} \ddot{g}(0,x) d\nu(x).
\]
By the Definition
\ref{def_G_f} 
of $G_{\widetilde{\gamma}(0)} (\widetilde{\gamma}(t))$, 
we have
\[
G_{\widetilde{\gamma}(0)} (\widetilde{\gamma}(t)) 
= 
\delta_{\eta}(\widetilde{\gamma}(t)) \int_{J_{\lambda_0}} \log|f'_t \circ \Psi_{\tilde{\gamma}(t)}|d\nu=
- \int_{J_{\lambda_0}} g(t,x)
d\nu (x).
\]
The assertion follows from the Definition
\ref{def_metric_G} 
of $||\cdot||_G$, after taking two derivatives in $t$ in the last expression.
\end{proof}

\begin{cor}\label{cor_G}
The following assertions
are equivalent:
\begin{enumerate}
\item $||\vec{v}||_G = 0$;
\item $||\vec{v}||_{\mathcal{P}} =0$;
\item ${\rm Var}(\dot{g}(0,x), \nu) = 0$;
\item $\dot{g}(0,x)$ is
a $C^0$-coboundary, i.e., it is $C^0$-cohomologous to zero.
\end{enumerate}
\end{cor}

\begin{proof}
The equivalence between the first three assertions
immediately follows
from Proposition \ref{prop_conf}. 
The equivalence between (3) and (4) follows from Lemma \ref{lem_variance_cobdry}.
\end{proof}

We conclude this section with the following lemma, which we will need to prove
that the forms introduced so far induce a path metric on $\Omega$. 

\begin{lem} \label{lem_K_equation}
If $||\vec{v}||_G=0$, then there exists a constant $K \in \bbR$ such that, for every $n\in\mathbb N$, we have 
\begin{align*}
\frac{d}{dt}\Big|_{t=0} 
S_n\big(
\log|f'_t\circ \Psi_{\widetilde{\gamma}(t)}(x)|\big)
&= K \cdot S_n
\big(
\log|f' \circ \Psi_{\widetilde{\gamma}(0)}(x)|\big)
\end{align*} for all 
$n$-periodic
points
$x$ of $f$ 
in $J_{\lambda_0}$.
Here $S_n\phi$ denotes the Birkhoff sum of $\phi$. 
\end{lem}

\begin{proof}
By the assumption on 
$||\vec{v}||_G$
and Corollary \ref{cor_G}, the derivative $\dot{g}(0,x)$ of the map $g(t,x)= -\delta(f_t)\log|f_t'\circ \Psi_{f_t}(x)|$ is a $C^0$-coboundary.  
Hence,
there exists a continuous 
function $h: J_{f_0}\to\mathbb{R}$ such that $\dot{g}(0,x)=h(x)-h(f_{0}(x))$
for every 
$x\in J_{f_0}$.
Let $x \in J_{f_0}$
be a $n$-periodic point of $f_{0}$. 
Then, we have
\begin{align*} 
0&= h(x) - h(f_{0}^n(x)) = h(x) - h(f_{0}(x)) + h(f_{0}(x)) - \cdots - h(f_{0}^n(x))\\
&= \frac{d}{dt}\bigg|_{t=0} g(t,x)+\frac{d}{dt}\bigg|_{t=0} g(t,f_{0}(x))+\cdots+\frac{d}{dt}\bigg|_{t=0} g(t,f_{0}^{n-1}(x))\\
&=-\frac{d}{dt}\bigg|_{t=0} \delta_{\eta}(f_t)S_n\big( \log|f'_t\circ \Psi_{f_t}(x)|\big).
\end{align*} 
Applying the chain rule and using the inequality
$\delta_\eta (f_0)>0$,
we obtain
\begin{equation*}
\frac{d}{dt}\Big|_{t=0} S_n
\big(\log|f'_t\circ \Psi_{f_t}(x)|\big)
= -\frac{\frac{d}{dt}\big|_{t=0} \delta_{\eta}(f_t)}{\delta_{\eta}(f_0)} \cdot S_n
\big(\log|f' \circ \Psi_{f}(x)|\big). 
\end{equation*}
Therefore, the assertion follows choosing
$K \defeq \frac{d}{dt}\big|_{t=0} \delta_{\eta}(f_t)/\delta_{\eta}(f_0)$.
\end{proof}

\section{The Hessian form defines a path metric}\label{sec_pathmetric}

We prove our main result, Theorem \ref{thm_main},
in this section. We first prove two preliminary results in Sections \ref{sec_ana_UTB}--\ref{sec_C1_paths} stating that $\langle \cdot, \cdot \rangle_G$ is real-analytic on the unit tangent bundle $UT\Omega$
of $\Omega$, and that any $C^1$-path in $\Omega$ has strictly positive length with respect to $\langle \cdot, \cdot \rangle_G$. We conclude the proof of Theorem \ref{thm_main} in Section \ref{sec_Pf_Thm}.

\subsection{Analyticity of $\langle \cdot, \cdot \rangle_G$ on the unit tangent bundle} \label{sec_ana_UTB}
We fix in this section a $\Lambda$-hyperbolic component $\Omega$ in a parabolic family $\Lambda$ in ${\rm rat}^{cm}_D$
and
we show that the bilinear form $\langle \cdot, \cdot \rangle_G$ is analytic on the unit tangent bundle $UT\Omega$ 
of $\Omega$. Recall that this form depends on a parameter $\eta \in (0,\log D)$, which will be fixed throughout this section.

\medskip

As in Section \ref{ss:P-omega}, 
we will 
think of $\Omega$ and
$\Lambda$
as subfamilies of
${\rm Rat}_D^{cm}$, by means of suitable lifts.
We fix $\lambda_0$ and only work in a 
sufficiently small neighbourhood of $\lambda_0$ in $\Omega$.
By means of suitable charts, we can then assume that $\lambda_0=0$
and that $D_\ell(0,R_0)\subset \Omega$, where $\ell$ is the complex dimension of $\Omega$ and $\Lambda$.

\medskip

Let $\{\vec{v}_s\}_{s\in D_1(0,R_0)}$ be a holomorphic family of elements 
of
$ \mathbb{C}^\ell \setminus\{\vec{0}\}$, 
i.e., we assume that the map $s \mapsto 
\vec{v}_s\in \mathbb C^\ell \setminus\{\vec{0}\}$ is holomorphic.
Observe also that we can
identify $T_\lambda \Omega$ with 
$\mathbb C^\ell$ 
for every $\lambda \in D_\ell (0,R_0)$.
Consider the map $\gamma\colon D_{\ell}
(0,R_0)
\times D_1(0,R_0)\times D_1 (0,R_0)\to \mathbb C^\ell$ given by
\[
\gamma (\lambda, t,s ) = \lambda + t\vec{v}_s.
\]
Up to shrinking $R_0$, we can assume that the image of $\gamma$ is contained in $\Omega$. 
Moreover, it is clear from the definition that $\gamma$ satisfies
\[
\gamma(\lambda,0, s) = \lambda 
\quad \mbox{ and }
\quad
\frac{d}{dt}\Big|_{t=0}
\gamma (\lambda, t, s) = \vec{v}_s \in T_{\lambda}{\Omega}
\quad \mbox{ for all } \lambda \in D_\ell (0,R_0)
\mbox{ and }
s \in D_1 (0,R_0).
\]

For every
$(\theta,\lambda, t,s) \in \bbR \times D_\ell(0,R_0)\times D_1(0,R_0)\times D_1(0,R_0)$, we also
define
$$\phi_{(\theta,\lambda, t,s)} \defeq -\delta_\eta(\lambda)\log|f'_\lambda| + \theta\log |f'_{\gamma(\lambda, t,s)}\circ\Psi_{\lambda, t,s}|:J_\lambda \to \bbR$$
where we denote by $\Psi_{\lambda, t,s}:J_\lambda \to J_{\lambda,t,s}$ the conjugacy map induced by the holomorphic motion on $\Omega$.

\begin{prop}\label{prop_metric_an}
There exists $0<R<R_0$ such that the map $(\delta_\eta(\lambda_0)-R,\delta_\eta(\lambda_0)+R)\times D_\ell(0,R)\times D_1(0,R)\times D_1(0,R)\ni (\theta,\lambda,t,s) \mapsto \mathcal{P}(\phi_{(\theta,\lambda,t,s)})$ is real-analytic.
\end{prop}

We will show Proposition \ref{prop_metric_an} in Section \ref{ss:proof-an}.
We first deduce the following
corollary, giving the analyticity of
the metric $||\cdot||_G$
on the tangent bundle of ${\Omega}$.
For every $\lambda \in \Omega$, we denote
by $\nu_\lambda$ the unique equilibrium state
of $-\delta_\eta(\lambda)\log|f'_\lambda|:J_\lambda \to \mathbb R$.

\begin{cor}\label{c:metric-analytic}
There exists $R\in (0,R_0)$ such that
the map 
$$D_\ell(0,R)\times
D_1(0,R)
\times
D_1(0,R)
\ni (\lambda, t,s) \mapsto G_0(\gamma(\lambda,t,s)) = \delta_\eta(\gamma(\lambda,t,s))\int_{J_\lambda} \log|f'_{\gamma(\lambda, t,s)} \circ\Psi_{\lambda, t,s}| d\nu_\lambda$$
is real-analytic. Moreover, the map 
$D_\ell(0,R)\times
D_1 (0,R) \ni (\lambda,s)\mapsto (G_{\lambda}''(\lambda))
(\vec{v}_s, \vec{v}_s)$ 
is real-analytic.
\end{cor}

\begin{proof}
By Propositions \ref{prop_deripressure} and \ref{prop_metric_an}, taking the first derivative with respect to $\theta$ of the pressure function  $\mathcal{P}(\phi_{(\theta,\lambda,t,s)})$,
for every $(\lambda, t,s)\in D_\ell (0,R)\times D_1 (0,R)\times D_1 (0,R)$
we have
$$\frac{d}{d\theta}
\Big|_{\theta=0}
\mathcal 
P(-\delta_\eta(\lambda)\log|f'_\lambda| + \theta\log |f'_{\gamma(\lambda, t,s)}\circ\Psi_{\lambda, t,s}|) = \int_{J_\lambda} \log|f'_{\gamma(\lambda, t,s)}\circ\Psi_{\lambda, t,s}| d\nu_\lambda.$$
Hence, again 
by Proposition \ref{prop_metric_an}, the map
$$D_\ell(0,R)
\times D_1 (0,R)
\times D_1 (0,R)
\ni (\lambda, t,s) \mapsto \int_{J_\lambda} \log|f'_{\gamma(\lambda, t,s)}\circ\Psi_{\lambda, t,s}| d\nu_\lambda$$ is real-analytic. By Proposition \ref{prop_deltaanalytic} and the definition of $\gamma(\lambda,t,s)$, the map 
\begin{equation*}\label{eq_5.2}
D_\ell(0,R)\times
D_1(0,R)
\times
D_1(0,R)
\ni (\lambda, t,s) \mapsto G_0(\gamma(\lambda,t,s)) = \delta_\eta(\gamma(\lambda,t,s))\int_{J_\lambda} \log|f'_{\gamma(\lambda, t,s)} \circ\Psi_{\lambda, t,s}| d\nu_\lambda
\end{equation*}
is real-analytic. This gives
the first assertion.

\medskip

Taking two derivatives in $t$ of the function $G_0(\gamma(\lambda,t,s))$ and evaluating 
at $t=0$, we see that the map
\begin{equation*}\label{eq_5.21}
D_\ell(0,R)\times
D_1(0,R)
\ni (\lambda, s) \mapsto \frac{d^2}{dt^2}\bigg|_{t=0} G_0(\gamma(\lambda,t,s)) = (G_\lambda''(\lambda))(\vec{v}_s,\vec{v}_s) 
\end{equation*}
is real-analytic. This completes the proof.
\end{proof}

\subsection{Proof of Proposition \ref{prop_metric_an}}
\label{ss:proof-an}
We again follow the
general strategy as the proof of
\cite[Theorem A]{SU10} or \cite[Theorem 9.3]{UZ04real}. As the arguments are similar to those of the proof of Proposition \ref{prop_analytic_Mis}, 
we will just sketch them.

\medskip

For every $z \in J_{0}$ and $\lambda \in \Omega$,
we denote
$z_\lambda \defeq \Psi_\lambda(z)$. For 
every
$z \in J_{0}$ and $(\lambda,t,s) \in
D_\ell({0},R_0)
\times
D_1({0},R_0)
\times 
D_1({0},R_0)$,
consider the map
$$\psi_{z}(\lambda, t,s) \defeq \frac{f'_{\gamma (\lambda, t,s)}\circ \Psi_{\lambda, t,s}(z_\lambda)}{f'_\lambda(z_\lambda)},$$
where we recall that
$\Psi_{\lambda, t,s}:J_\lambda \to J_{\lambda, t,s}$ 
is the conjugacy map. 
Up to 
shrinking $R_0$ if necessary, for all $z \in J_{0}$ 
and $(\lambda,t,s)\in D_\ell(0,R_0)
\times D_1({0},R_0)
\times D_1({0},R_0)$,
we have
$|\psi_{z}(\lambda,t,s)-1|<1/5$.
Then, 
for every $z \in J_{0}$,
there exists a branch of $\log\psi_z$ 
sending $0$ to $0$ and whose 
modulus is bounded by $1/4$.
By Lemma \ref{lem_6.3}, up to further shrinking $R_0$, the analytic
map
$\Re\log \psi_{z} : D_\ell (0,R_0)
\times D_1({0},R_0)
\times D_1({0},R_0)\to \mathbb{R}$ has an analytic extension $\Re \widetilde{\log \psi_{z}} : 
D_{2\ell} (0,R_0)
\times D_2({0},R_0)
\times D_{2}({0},R_0)
\to \mathbb{C}$. 
Recall that $D_\ell(0,R_0)$
(resp. $D_\ell(0,R_0)$)
is seen
as a subset of the points
of  
$\mathbb C^{2\ell}$
(resp. $\mathbb C$)
with real coordinates
by means of the immersion $\iota_\ell$
(resp. $\iota_1$)
as in \eqref{eq_iota}, and that 
we have
$\iota_\ell (0)=0\in \mathbb C^{2\ell}$
and
$\iota_1 (0)=0\in \mathbb C^{2}$.

\medskip

For $(\theta,\lambda, t,s) \in \bbC \times 
D_{2\ell} (0,R_0)
\times D_2({0},R_0)
\times D_{2}({0},R_0)$, consider the map 
$\zeta_{(\theta,\lambda, t,s)}\colon J_0 \to \mathbb C$ given by
\begin{align*}
\zeta_{(\theta,\lambda, t,s)}(z) &\defeq -\delta_\eta(\lambda)\log|f'_\lambda(z_\lambda)|-\theta \Re\widetilde{\log \psi_{z}}(\lambda,t,s) + \theta\log|f'_\lambda(z_\lambda)|\\
&= -\theta \Re \widetilde{\log \psi_{z}}(\lambda,t,s)+(\theta-\delta_\eta(\lambda))\log|f'_\lambda(z_\lambda)|.
\end{align*}

Let $\beta$
be such that the conjugacy maps
$\Psi_\lambda : J_{0} \to J_\lambda$
and
 $\Psi_{\lambda,t,s} : J_{\lambda} \to J_{\lambda,t,s}$ are
 $\beta$-H\"older continuous 
 for all $\lambda \in D_\ell(0,R_0)$
and all $(\lambda,t,s) \in D_\ell(0,R_0)\times D_1(0,R_0) \times
D_1 (0,R_0)$, respectively.
The map $(\theta,\lambda,t,s) \mapsto \zeta_{(\theta,\lambda,t,s)}$ is then 
continuous with respect to the $\beta$-H\"older norm.
Applying 
Lemma \ref{l:extension} to the
families
$\Re \zeta_{(\theta,\lambda,t,s)}$
and 
$\Im \zeta_{(\theta,\lambda,t,s)}$,
up to shrinking $R_0$, 
gives an extended family
$\mathbb{C}\times 
D_{2\ell} (0,R_0)
\times D_2({0},R_0)
\times D_{2}({0},R_0)
\ni (\theta,\lambda,t,s) \mapsto \widetilde{\zeta}_{(\theta,\lambda,t,s)} : \mathbb P^1 \to \mathbb C$ which is
continuous with respect to the $\beta/2$-H\"older norm and satisfies
$ \widetilde{\zeta}_{(\theta,\lambda,t,s)} (z)=
{\zeta}_{(\theta,\lambda,t,s)} (z)$
for every 
$(\theta,\lambda,t,s) \in \mathbb{C} \times \iota_{\ell+2}^{-1}
(D_{2\ell}(0,R_0)
\times 
D_{2}(0,R_0)
\times
D_{2}(0,R_0))$
and $z\in J_0$. 

\medskip

Fix a norm 
$||\cdot||_\diamond$ 
as in Section \ref{ss:operators}
such that $||\cdot||_\diamond \lesssim ||\cdot||_{C^{\beta/2}}$ 
and,
for $(\theta,\lambda,t,s) \in \mathbb{C}\times
D_{2\ell}(0,R_0)
\times 
D_{2}(0,R_0)
\times
D_{2}(0,R_0)$,
consider the complex transfer operator $\mathcal{L}_{(\theta,\lambda,t,s)} \defeq \mathcal{L}_{\widetilde \zeta_{(\theta,\lambda,t,s)}}$.
As in Lemma \ref{lem_op_hol},
there exists $0<R_1<R_0$
such that the map $D_1(\delta_\eta(0),R_1) \times 
D_{2\ell}(0,R_0)
\times 
D_{2}(0,R_0)
\times
D_{2}(0,R_0)
\ni (\theta,\lambda,t,s)\mapsto \mathcal{L}_{\widetilde \zeta_{(\theta,\lambda, t,s)}} \in L(B_\diamond (J_0))$
is holomorphic. 
As for
Proposition 
\ref{prop_analytic_Mis},
Proposition \ref{prop_metric_an} then follows 
from Kato-Rellich
perturbation theorem.

\subsection{Non-degeneracy along $C^1$ paths} \label{sec_C1_paths}
We fix in this section a {\it bounded}
$\Lambda$-hyperbolic
component $\Omega$ of a parabolic subfamily 
$\Lambda$ in ${\rm poly}_D^{cm}$. 
Using the metric $||\cdot||_G$, given a $C^1$ path $\gamma \colon (0,1)\to \Omega$, we define the {\it length} $\ell_G(\gamma)$ of $\gamma$ as
$$\ell_G(\gamma) \defeq \int_0^1 \|\gamma'(t)\|_G dt.$$
The main result of this section is the following proposition,
which states that the metric assigns a positive length to any (non-trivial) $C^1$ path in $\Omega$. The assumptions
on the polynomial family and on the boundedness of $\Omega$
in our main theorem will be used in the proof of this result.

\begin{prop}\label{p:non-deg-analytic-paths}
We have $\ell_G(\gamma)>0$ for any non-trivial
$C^1$ path $\gamma \colon (0,1)\to \Omega$.
\end{prop}

Observe that, if $\ell_G(\gamma)=0$, we must have $\|\gamma'(t)\|_G=0$ for almost every $t\in (0,1)$
(and hence, by continuity, for all $t\in (0,1)$).
For simplicity, we will say that the metric is
\emph{degenerate along
$\gamma$} if this happens.
The following lemma, which follows immediately from Lemma \ref{lem_K_equation}, characterizes when such a situation can occur.

\begin{lem}
Let
$\gamma\colon (0,1)\to \Omega$ 
be a $C^1$ path. Then 
the metric $||\cdot||_G$ is degenerate along $\gamma$ if and only if for
every $t \in (0,1)$ and for every $n \in \mathbb{N}$, we have
\begin{equation*}
\frac{d}{dt} \Big|_{t=0}
S_n \big(
\log |f'_{\gamma(t)} (x(\gamma(t)))|\big)
= K(\gamma(t)) 
S_n \big(\log |f'_{\gamma(t)}(x(\gamma(t)))|\big)
\end{equation*}
for every repelling $n$-periodic
point $x(\gamma(t))$, where $K(\gamma(t)) \defeq 
\delta_{\eta} (\gamma(0))^{-1}
\frac{d}{dt}\big|_{t=0} \delta_{\eta}(\gamma(t))$.
\end{lem}

The following corollary is an immediate consequence of the previous lemma.
\begin{cor}\label{c:formulas-deg}
Let
 $\gamma \colon (0,1)\to \Omega$ be a $C^1$ path along which the metric $||\cdot||_G$ degenerates. Then the following assertions hold:
\begin{enumerate}
\item
for every repelling $n$-periodic
point $x$ and $t_1,t_2\in (0,1)$, 
we have
\[
S_n \big(
\log |f'_{\gamma(t_2)} (x(\gamma((t_2))) |
\big)
=
S_n
\big(
\log  |f'_{\gamma(t_1)} (x(\gamma(t_1))) |
\big)
\cdot e^{\widetilde K(t_2, t_1)},
\]
where $\widetilde K(t_2, t_1)= \int_{t_1}^{t_2} K(\gamma(t))dt$;
\item for every pair of motions of 
repelling $n$-periodic 
points $x_i, x_j$, there exists a positive
constant $a_{i,j}$ such that
\[
S_n \big(
\log |f'_{\gamma(t)} (x_i (\gamma(t)))|\big)
=
{a_{i,j}}
S_n
\big(
\log |f'_{\gamma(t)} (x_j (\gamma(t)))|\big)
\]
for every $t$ in $(0,1)$.
\end{enumerate}
\end{cor}

We can now
prove Proposition \ref{p:non-deg-analytic-paths}.

\begin{proof}[Proof of Proposition \ref{p:non-deg-analytic-paths}]
Suppose by contradiction that we have
$\ell_G(\gamma)=0$. Then, 
 the metric must degenerate along $\gamma$. 
We denote by $\{x_i(\lambda)\}_{i\ge1}$ the set of maps parametrizing the repelling periodic points on $\Omega$, and let $n_i$ be the period
of the corresponding cycle.

\medskip

Fix $s_0\in (0,1)$ and an index $i_0$.
It follows from Corollary \ref{c:formulas-deg} that, for every $i\ge 1$
and every $s\in (0,1)$, we have
\begin{equation*}
S_{n_{i_0} n_i} \big(
\log |f'_{\gamma(s)} (x_i (\gamma(s)))|
\big)
= a_{i,i_0} e^{\widetilde K (s,s_0)} 
S_{n_{i_0} n_i}
\big(
\log |f'_{\gamma(s_0)} (x_{i_0} (\gamma(s_0)))|\big)
\end{equation*}
for some strictly
positive
(as both $x_i(\gamma(s))$ and $x_{i_0}(\gamma(s_0))$
are repelling)
constants $a_{i,i_0}$.

\medskip

As $\Omega$ is bounded,
we have $L(\gamma(s))\equiv \log D$ on $\Omega$ by Lemma \ref{lem_homology}.
Hence, it follows from 
the equidistribution of periodic points with respect to the measure of maximal entropy \cite{Lyubich82,Lyu83entropy}
that,
 for every $s \in (0,1)$,
 we  have
\begin{equation} \label{eq_tildeL}
\begin{aligned}
\log D
& =
\lim_{n\to \infty}
\frac{1}{n_{i_0} n} \frac{1}{D^{n_{i_0}n}}
\sum_{x_i\colon n_i =n_{i_0}n}
 a_{i,i_0} e^{\widetilde K (s,s_0)} 
 S_{n_{i_0} n} \big(
 \log |f'_{\gamma(s_0)} (x_{i_0} (\gamma(s_0)))|\big) \nonumber\\
& = e^{\widetilde K (s,s_0)} \cdot
  \lim_{n\to \infty}
 \left(
 \frac{1}{D^{n_{i_0}n}}
 \sum_{x_i\colon n_i=n}
a_{i,i_0} 
\cdot
\frac{1}{n_{i_0}n}
S_{n_{i_0}n} \big(
\log |f'_{\gamma(s_0)} (x_{i_0} (\gamma(s_0)))|\big)
\right)\\
& = e^{\widetilde K (s,s_0)} \cdot
 \left(
 \lim_{n\to \infty}
 \frac{1}{D^{n_{i_0}n}}
 \sum_{x_i\colon n_i=n}
a_{i,i_0} 
\right)
\cdot
\frac{1}{n_{i_0}}
S_{n_{i_0}} \big(
\log |f'_{\gamma(s_0)} (x_{i_0} (\gamma(s_0)))|\big),
\end{aligned}
\end{equation}
where in the last step we used the identity
\[S_{n_{i_0}n} \big(
\log |f'_{\gamma(s_0)} (x_{i_0} (\gamma(s_0)))|\big) =
n
S_{n_{i_0}} \big(
\log |f'_{\gamma(s_0)} (x_{i_0} (\gamma(s_0)))|\big).\]
We deduce that the function $\widetilde K(s,s_0)$ is independent of $s$.
By Corollary \ref{c:formulas-deg} (1), this shows that the absolute values of all the multipliers
of the $x_i (\gamma(s))$'s are constant along $\gamma$. 
However, this is impossible
as, by \cite[Theorem 8.25]{JiXie23}, there are only finitely many conjugacy classes of rational maps, not in the locus of (conjugacy classes of) flexible Latt\`es maps, having the same
set of absolute values of repelling multipliers. 
Therefore $\ell_G(\gamma) = 0$ and the proof is complete.
\end{proof}

\begin{rmk}
The conclusion of the proof of Proposition \ref{p:non-deg-analytic-paths} can be achieved also without making use of 
 \cite{JiXie23}. Indeed, assume as above that
  the absolute values of all the multipliers
of the $x_i (\gamma(s))$'s are constant along $\gamma$. 
For every $i$,
the absolute value of
the multiplier
of $x_i (\lambda)$ 
is an
harmonic 
function
on $\Omega$.  
This gives 
a family of harmonic functions with a non-trivial common level set
(possibly corresponding to a different real value, larger than 1, for each function)
which, by the above, must contain the image of $\gamma$.
Up to reparametrization, as $\Lambda$ is algebraic, we can also assume that $\Lambda=\mathbb C^\ell$, where $\ell ={\rm dim}_{\mathbb C} \Lambda$.
 It follows from 
 the maximum principle that this common level set 
 cannot be bounded in $\Lambda$
(as, otherwise, all the harmonic functions
 would be constant in the region bounded by the level set, hence constant there, hence on $\Lambda$, which contradicts the choice of $\Omega$ as before).
Hence, all of the repelling points stay repelling, with constant modulus of their multiplier, along some path going to infinity in $\Lambda=\mathbb C^\ell$.
This implies that, for every $\lambda $ belonging to
this path, we
have $L(\lambda)=\log D$. This contradicts the fact that, outside of a compact subset of $\Lambda$, all critical points escape to infinity, which gives $L(\lambda)>\log D$ by the Przyticki formula.
\end{rmk}

\subsection{Proof of Theorem \ref{thm_main}} \label{sec_Pf_Thm}
In this section we conclude the proof of our main theorem. 
We will use
the following theorem by Mityagin \cite{Mit15}, describing the zero set of a non-trivial
real analytic function.

\begin{thm}\label{t:mityagin}
Let $\mathcal{O} \subset \mathbb R^{l}$ 
be an open set and $\psi\colon \mathcal{O} \to \mathbb R$ an analytic function not identically vanishing. Then, the set $\{\psi=0\}$
is covered by a countable union of 
(not necessarily closed)
analytic submanifolds of $\mathcal{O}$.
\end{thm}

In \cite{Mit15},
the author
only states that, under the assumptions of the theorem,
the set $\{\psi=0\}$ has
zero Lebesgue measure. However, the proof is constructive
and, as already noted in 
\cite[Remark 5.23]{Kuchment16}, gives a decomposition of this set as a countable union of smooth submanifolds {of real co-dimension at least 1}.
As each of these submanifolds 
(obtained by means 
of the implicit function theorem)
is in 
the common
zero locus of analytic functions
(namely $\psi$ and some of its derivatives),
one can see that
these sets are indeed analytic submanifolds.

\medskip

We now conclude the proof of our main result; that is, we need to show that the function $d_G: \Omega \times \Omega \to \mathbb R$ given by
\begin{equation*}
d_G(x,y) \defeq \inf_\gamma \int_0^1 \|\gamma'(t)\|_G dt
\end{equation*}
is a distance function. Recall that
the infimum is taken over all the $C^1$-paths $\gamma$ connecting $x$ to $y$ in $\Omega$.

\begin{proof}[Proof of Theorem \ref{thm_main}]
Let $\Lambda$ and $\Omega$ be as in the statement. Observe that $d_G$ is a pseudo-metric; namely, we have 
$d_G(x,x) = 0$
and $d_G(x,y)=d_G(y,x)$
for every $x,y\in \Omega$,
and $d_G$ satisfies the triangle inequality. We need to show $d_G(x,y)>0$ for $x \neq y \in \Omega$.

\medskip

By Corollary \ref{c:metric-analytic}, the pseudo-metric $d_G$ induced by the Hessian form is described by a collection of positive semi-definite bilinear forms $A^1_\lambda$ on $T_\lambda\Omega$,
depending analytically 
on the point $\lambda\in \Omega$. Denote by $\mathscr{D}_1(\lambda)$ the determinant 
of the matrix representing
$A^1_\lambda$. By Corollary \ref{c:metric-analytic}, $\mathscr{D}_1:\Omega \to \mathbb R$ is real-analytic.
It is clear that the pseudo-metric $d_G$
is indeed a metric outside of the zero locus $S_1$
of the analytic map $\mathscr{D}_1 : \Omega \to \mathbb R$.
Hence, it is enough to prove that $d_G(x,y)>0$ for every $x,y\in S_1$. 

\medskip

We first show 
that we cannot have $S_1=\Omega$. Observe that  $S_1=\Omega$ means that at every $\lambda\in\Omega$ there is at least one degenerate direction for the metric. For $j=1,\ldots,2\ell$ where $\ell\defeq {\rm dim}_{\mathbb C} \Omega$, set 
$$U_j \defeq \{\lambda \in \Omega : \exists \vec{v}_1,\ldots, \vec{v}_j \in T_\lambda\Omega \text{ linearly independent with } \|\vec{v}_1\|_G = \ldots = \|\vec{v}_j\|_G =0\}.$$ 
It is straightforward to see that $S_1 = U_1 \supseteq U_2 \supseteq \cdots \supseteq U_{2\ell}$. There are two cases to consider.

\medskip

{\bf Case 1:} 
$U_{j+1} \neq U_{j}$ for some $j\in \{1,\ldots,2\ell-1\}$.
Let $j^\star$ be the minimum $j$ satisfying this property. Then $U_{j^\star} \setminus U_{j^\star+1}$ contains an open set $\mathcal{A}$ of $\Omega$.
By the definition of $\mathcal{A}$ and the analyticity of the metric,
there exists a
$j^\star$-dimensional
subbundle $V\subset T\Omega$ on $\mathcal A$ such that,
for every $\lambda \in \mathcal{A}$,
$\|\cdot\|_G$ is degenerate on the fiber
$V_\lambda$.
Consider a $C^1$ path $\gamma\colon  (0,1)\to \mathcal A$
whose tangent $\gamma'(t)$
is contained in $V_{\gamma(t)}$ 
for every $t \in (0,1)$. By construction,
we have $\ell_G(\gamma)=0$, contradicting Proposition \ref{p:non-deg-analytic-paths}.

\medskip

{\bf Case 2:} If $U_{j+1} = U_{j}$ for all $j = 1,\ldots, 2\ell-1$, then $U_{2\ell} = U_1 = \Omega$, meaning that we have $\|\cdot\|_G \equiv 0$ on the tangent bundle $T\Omega$. Then any $C^1$-path in $\Omega$ has length 0, contradicting Proposition \ref{p:non-deg-analytic-paths}.  

\medskip

Therefore, 
as $S_1\neq \Omega$,
the function $\mathscr{D}_1: \Omega \to \mathbb R$ is not identically zero on $\Omega$ and, up to working locally,
we can apply Theorem \ref{t:mityagin} to $\mathscr{D}_1: \Omega \to \mathbb R$.
It follows that there is a countable collection $\{S_1^j\}_{j\ge1}$ of connected analytic submanifolds of $\Omega$ 
of real co-dimension at least 1
covering $S_1$. It is enough to show that
$d_G(x,y)>0$ for any two distinct points $x$ and $y$ belonging to the same 
real co-dimension one submanifold in the collection
 $\{S_1^j\}_{j\ge1}$, say $S_1^1$.

\medskip

If ${\rm dim}_{\mathbb R} \Omega=2$, the proof is complete by Proposition \ref{p:non-deg-analytic-paths}. Indeed,
as the submanifold $S_1^1$ is smooth and one-dimensional, it itself gives a smooth path joining $x$  and $y$ in $S_1^1$.
As the length of this path is strictly positive by Proposition \ref{p:non-deg-analytic-paths}, the proof in this case is complete.

\medskip

We now treat the general case where ${\rm dim}_{\mathbb R}\Omega > 2$. By the above argument, $S_1^1$ is an analytic submanifold, and we need to show that $d_G(x,y)>0$ for any pair of distinct points $x,y\in S_1^1$. Observe that any path between $x$ and $y$ not contained in $S_1^1$ must necessarily have a positive length. Hence, we can restrict ourselves to
paths with are contained in $S_1^1$, whose length
can be computed by considering the restrictions 
of the pseudo-metric represented by 
$A^1_\lambda$ to the 
(real) 
tangent spaces
$T_\lambda S_1^1$. 
Let $\mathscr{D}_2 (\lambda)$ be the determinant of the associated matrix of $A^2_\lambda$ on the tangent space $T_\lambda S_1^1$. Then $\mathscr{D}_2 : S_1^1 \to \mathbb R$ is an analytic function.
We argue as above that the zero locus $S_2$
of $\mathscr{D}_2 :S_1^1 \to \mathbb R$ is not equal to $S_1^1$, and is covered by
a countable collection $\{S_2^j\}_{j\ge1}$ of connected analytic submanifolds of $S_1^1$. As before, we can then assume that $x$ and $y$ belong to the same component $S_2^1$ of $S_2$, and that any continuous
path of trivial length joining $x$ and $y$ must be contained in $S_2^1$, which is an analytic submanifold of $\Omega$ of real dimension ${\rm dim}_{\mathbb R}\Omega-2$.

\medskip

Working by induction, we see that
any continuous path of trivial length between $x$ and $y$ must be contained in an analytic 
real one-dimensional
submanifold $S_{{\rm dim_{\mathbb R} \Omega -1}}^1$,  where $S_{j+1}^1$
is a component of the singular locus of the
restriction of the pseudo-metric 
to $S_j^1$. 
The conclusion
now follows from 
Proposition \ref{p:non-deg-analytic-paths}, as in the case where ${\rm dim}_{\mathbb R} \Omega=2$. The proof is complete.
\end{proof}

\printbibliography

\end{document}